\documentclass[final,onefignum,onetabnum]{siamart250211} 

\usepackage{amsmath,amssymb,amsfonts}
\usepackage{graphicx,subcaption}
\usepackage{hyperref}
\usepackage{mathrsfs}
\usepackage{bm}
\usepackage{xfrac}
\usepackage{color}
\usepackage{enumerate}
\usepackage{stmaryrd}

\usepackage{lipsum}
\usepackage{amsfonts}
\usepackage{graphicx}
\usepackage{epstopdf}
\usepackage{algorithmic}
\ifpdf
  \DeclareGraphicsExtensions{.eps,.pdf,.png,.jpg}
\else
  \DeclareGraphicsExtensions{.eps}
\fi

\newsiamremark{remark}{Remark}
\newsiamremark{hypothesis}{Hypothesis}
\crefname{hypothesis}{Hypothesis}{Hypotheses}
\newsiamthm{claim}{Claim}
\newsiamremark{fact}{Fact}
\crefname{fact}{Fact}{Facts}

\author{
Ehsan Faghihifar\thanks{Independent Researcher, Tehran, Iran 
(\email{eh.faghihi@gmail.com}).}
}

\title{Boundedness of Left Half-Plane Eigenvalues for Coefficient-Coupled Sturm--Liouville Problems with Application to Fourier Modal Methods}

\begin{document}
\maketitle

\begin{abstract}
We study a class of Sturm--Liouville problems of the form
\[
-(p\,y')' + q\,y = \lambda\, w\, y,
\]
on a finite interval with complex-valued coefficients, where $p$ and $w$ are piecewise smooth, the ratio $w/p$ is real and positive, and $q$ is bounded. We prove that all eigenvalues in the open left half-plane are contained in a bounded set, which implies that only finitely many eigenvalues lie in this region. This stands in contrast to the known unboundedness, for real-valued coefficients, when $p$ or $w$ changes sign independently. A canonical instance of this class, with $w=p$, arises in transverse-magnetic (TM) diffraction by metallic lamellar gratings, a benchmark problem in computational photonics, central to the development of modal methods.
In Fourier modal methods, in particular, the emergence of spurious modes with unbounded propagation constants (eigenvalues), rooted in discretization of sign-changing coefficients, leads to notorious convergence difficulties. These modes cannot be excluded \textit{a priori}, since genuine eigenvalues are not constrained by conventional bounds in this regime. Nevertheless, our result shows that the physical eigenvalues remain bounded, providing a rigorous criterion for identifying spurious modes in low-loss metallic gratings.
\end{abstract}


\section{Introduction}
\label{sec:intro}

The analysis of wave propagation and diffraction in periodic or layered media is a central problem in modern photonics. Lamellar gratings, as one-dimensional periodic structures, are particularly significant, both as canonical numerical models and key photonic components~\cite{Petit2013Book, Bonod2015, Popov2014Book}. 
Among various numerical techniques devised to address this problem, modal methods are especially prominent, owing to their physical insight and computational efficiency~\cite{Wexler1967, Snyder1983Book, Li1997, Kocabas2009}. These methods rely on expressing the electromagnetic fields as a superposition of eigenmodes of a \emph{Sturm--Liouville} (SL) differential operator arising from Maxwell’s equations, whose spectral properties, such as selfadjointness, completeness, eigenvalue distribution, and boundedness, form the mathematical foundation of the framework~\cite{Zettl2005Book, Hanson2013Book}.

Early applications of modal methods in grating analysis relied on Fourier series to describe permittivity functions and electromagnetic fields~\cite{Burckhardt1966, Kasper1973, Knop1978}. Later, the analytic modal method (AMM) introduced a new approach for gratings with piecewise constant permittivity profiles, determining layer eigenmodes in closed form and enforcing the interface conditions analytically~\cite{Botten1981a, Botten1981b, Botten1981c, Botten1985, Sheng1982, Li1993b}. While conceptually straightforward, AMM is restricted to simple geometries and requires solving transcendental characteristic equations in the complex plane. Subsequently, further numerical modal methods were developed. Most notably, inspired by the early Fourier modal implementations and an equivalent approximate method, called the coupled-wave method~\cite{Kogelnik1969, Magnusson1978}, a more rigorous, generalized, standard Fourier-based modal approach took shape, known as the rigorous coupled-wave analysis (RCWA) or the Fourier modal method (FMM)~\cite{Moharam1986, Moharam1982, Moharam1983, Moharam1995a, Li1997, Li1999, Li2001Chap}. We use the latter term throughout for the full range of Fourier-based modal formulations. To date, FMM remains a popular approach, owing to its physical insight, exceptional versatility, and computational efficiency.

Nonetheless, modal methods and particularly FMM, exhibit well-known far-field convergence difficulties for transverse-magnetic (TM) polarization in metallic gratings, in stark contrast to the well-behaved transverse-electric (TE) case~\cite{Li1993a}. Although the advent of Fourier factorization rules~\cite{Garnet1996, Lalanne1996, Li1996a} helped circumvent truncated convolution errors (TCE)~\cite{Faghihifar2022}, anomalous convergence persisted with low-loss or lossless metallic gratings. Depending on implementation, these numerical anomalies have been attributed to a variety of seemingly distinct but cognate sources, including the Gibbs phenomenon~\cite{Li1993a, Kim2012}, TCE~\cite{Li1996a, Lalanne1996}, ill-conditioning of Fourier coefficient matrices~\cite{Popov2004a}, violation of Li’s inverse rule~\cite{Popov2004a, Watanabe2006, Khavasi2008}, field singularities at metal--dielectric edges~\cite{Gundu2010a, Li2011, Li2012, Mei2014}, omission of \emph{ghost} modes~\cite{Foresti2006}, and most notably, the emergence of \emph{spurious} modes with unbounded propagation constants~\cite{Lyndin2007}. To that effect, numerous remedies have been proposed, including finite-difference or adaptive-resolution schemes~\cite{Lalanne2000, Garnet1999, Garnet2010, Guizal2009}, subsectional spectral or pseudo-spectral approaches~\cite{Morf1995, Edee2011, Edee2015, Randriamihaja2016, Chiou2009, Song2011, Song2013, Garnet2014}, least-squares filtering~\cite{Gundu2010a, Gundu2010b}, artificial loss~\cite{Popov2004a}, and non-modal Fourier methods~\cite{Khavasi2007, Khavasi2009}. Notwithstanding, fundamental questions regarding the nature of spurious modes remain largely unanswered.

While regarded as numerical artifacts, the noted difficulties are also rooted in the spectral structure of the underlying problem, namely a non-selfadjoint indefinite SL equation. Several studies have shown that the associated spectrum exhibits unexpected features, such as modes with propagation constants exceeding the bound typically associated with the maximum dielectric constant~\cite{Tishchenko2005, Foresti2006}, and \emph{ghost} modes with non-real squared propagation constants (i.e., eigenvalues) even in lossless settings~\cite{Foresti2006, Sheng1982, Sturman2007a, Sturman2007b}. Neglecting these modes leads to pronounced convergence defects~\cite{Foresti2006}. Whether such modes are infinite in number appears to be an open question. Furthermore, the TM eigenmodes are not mutually orthogonal in the usual sense; consequently, modal expansions rely on a biorthogonal system, as established in non-selfadjoint spectral theory~\cite{Kocabas2009, Botten1981b, Hanson2013Book, Naimark1968Book}.

A general SL problem over a finite interval seeks nontrivial solutions to 
\begin{equation}\label{eq:SL_general_intro}
-(p\,y')' + q\,y = \lambda\, w\,y,
\end{equation}
where $p,q,w$ are complex-valued functions, $y$ is the eigenfunction, and $\lambda$ is the eigenvalue. The TM problem in gratings lies in this framework, with the magnetic field component $H_y$ as the eigenfunction, $w=p=\epsilon(x)^{-1}$, $q=-k_0^2$, and $\lambda=-\beta^2$, where $k_0$ is the vacuum wave number, $\beta$ is the longitudinal propagation constant, and $\epsilon(x)$ is the permittivity profile, i.e., typically modeled as a piecewise constant periodic function which can change sign for metallic gratings and become complex-valued in the presence of losses.  A similar formulation applies to TM modes in waveguides~\cite{Kocabas2009}.

In classical \emph{definite} SL theory, where all coefficients are real-valued, $p>0$, $w>0$ (a.e.), and $p^{-1},q,w \in L^1(a,b)$, the associated operator is selfadjoint in the weighted Hilbert space $L^2_w(a,b)$. It is bounded from below, with a discrete spectrum of real eigenvalues and a complete orthogonal system of eigenfunctions~\cite{Zettl2005Book, Weidmann1987Book}. Allowing the potential $q$ to be complex while keeping $p>0$ and $w>0$ (a.e.) can lead to complex eigenvalues~\cite{Hilb1911}. If $p$ changes sign while $w>0$ (a.e.), the spectrum remains real under the associated Hilbert space framework, but can become unbounded from both below and above, provided the subsets on which $p$ is positive and negative have nonzero Lebesgue measure~\cite{Moller1999}. A similar unboundedness result holds if $w$ changes sign, under additional assumptions~\cite{Kong2003}. A sign change in $w$ is the characteristic condition for the classical \emph{indefinite} SL problems~\cite{Zettl2005Book, Weidmann1987Book}. Here, Kreĭn space theory provides a natural framework for the analysis~\cite{Curgus1989}.

Beyond the definite setting, spectral properties can indicate additional complexities. Early studies proved the existence and finiteness of non-real eigenvalues for classical indefinite problems~\cite{Haupt1914, Richardson1918, Mingarelli1982, Mingarelli2011}. More recent works have established \emph{a priori} bounds on their real and imaginary parts under various regularity assumptions~\cite{Behrndt2013a, Behrndt2013b, Behrndt2013c, Behrndt2018, Qi2014, Qi2016, Kikonko2016, Xie2013, Xie2017}. The difference with the definite case can be even more subtle concerning eigenfunction completeness. In the sense of the \emph{Riesz basis property}, i.e., whether the eigenfunctions of the indefinite SL operator form a Riesz basis for the associated Hilbert space with $|w|$ as the weight function, the outcome can be non-trivial. While under certain conditions completeness is proven~\cite{Kaper1984, Beals1985, Curgus2013}, the Riesz basis property does not hold in general~\cite{Volkmer1996}, and explicit counterexamples have been constructed~\cite{Pyatkov1992, Fleige1998, Abasheeva1997, Binding2004}.

In the TM metallic grating setting, the governing equation falls into a largely unexplored regime of indefinite SL problems, in which the principal coefficient and the weight are not independent of one another. Classical studies in this area address situations in which the indefiniteness arises through either the weight function or the principal coefficient in a partially decoupled manner. In contrast, the regime considered here is distinguished by a direct coupling of the two: the ratio $w/p$ is required to be positive, so that they share their sign and phase pointwise. To the best of our knowledge, existing results provide only partial guidance on spectral behavior in such coupled settings. In addition, losses render the coefficients complex-valued, so that the problem can also be non-selfadjoint, although the lossless case is the more severe one numerically. This combination of coefficient coupling, indefiniteness, and non-selfadjointness defines a non-classical spectral landscape, which also has important consequences for the stability of numerical modal methods. The TM grating problem is the special case $w=p$.

In this paper, motivated by the TM problem, we consider equation\eqref{eq:SL_general_intro} and establish boundedness of eigenvalues for a relatively general class in which $p$ and $w$ are piecewise in $W^{2,\infty}$, non-vanishing, and admit finitely many discontinuities, with $w/p$ real and positive, while $q$ is bounded. Writing $w=r^2p$ with $r>0$, we further assume continuity of the state vector across each interface $x_j$, the non-degeneracy condition $(p\,r)(x_j^+) + (p\,r)(x_j^-) \neq 0$ at each interface, and pseudo-periodic, Dirichlet, or Neumann boundary conditions. We show that these assumptions imply boundedness of the eigenvalues in the left half-plane, in contrast to the decoupled cases described above, where the eigenvalues are unbounded in both directions. Classical SL theory then yields that only finitely many eigenvalues can lie in this region, and the exclusion region is shown to admit an explicit bound in terms of the coefficients and the interface data. The proof relies only on classical ODE theory and uniform estimates for the local solutions; it requires no operator-theoretic framework, in particular no realization in a Hilbert or Kreĭn space. For piecewise $C^2$ coefficients, which cover the configurations of physical interest, the regularity considerations involved become automatic and the argument is unchanged.

From a computational photonics perspective, this result formalizes a long-standing numerical observation in Fourier modal methods and provides a simple yet rigorous criterion for distinguishing spurious modes. In the TM lamellar grating problem, this can be summarized as follows: as the truncation order increases, modes with $\operatorname{Re}(\beta^2)\to +\infty$ are identified as spurious and have no physical analogue. One should note that, due to the indefinite nature of the governing equation, physical eigenvalues may not be bounded by classical estimates such as $|\beta| < k_0\sqrt{\epsilon_{\max}}$~\cite{Tishchenko2005, Foresti2006}.

The structure of this paper is as follows. In Section~\ref{sec:physical}, we recall the TE and TM scalar formulations and their relation to the general SL form. In Section~\ref{sec:spectral}, we present and prove the main theorem on left half-plane eigenvalue boundedness, together with a proposition giving an explicit bound, a corollary for additional boundary conditions, a finiteness remark, and a remark extending the result to a coefficient ratio of constant phase. In Section~\ref{sec:numerics}, applications of the results to grating analysis are illustrated through numerical examples, followed by the conclusion in Section~\ref{sec:conclusion}.


\section{Physical Motivation}
\label{sec:physical}

In this work, gratings are non-magnetic media, typically piecewise constant, that are periodic in the transverse plane and longitudinally invariant. In the context of gratings, eigenvalue problems arise in two distinct formulations of the same governing equation, each associated with different spectral properties. Consider the general vectorial wave equation for the electric or magnetic fields
\begin{equation}\label{eq:wave_equation}
\nabla \times \nabla \times \bm{E} = k_0^2 \epsilon \bm{E},\qquad \nabla \times ( \epsilon^{-1} \nabla \times \bm{H} ) = k_0^2 \bm{H}.
\end{equation} 
where $k_0 = {\omega}/{c}$ is the vacuum wave number for $\omega$ denoting the angular frequency and $c$ the speed of light. When the wave propagation is assumed in the transverse plane and \eqref{eq:wave_equation} is solved for the frequency as the eigenvalue with given transverse components, e.g., in photonic band structure calculations, the resulting eigenvalue problem is selfadjoint~\cite{Joannopoulos2008Book}. However, when the propagation has a longitudinal component, the transverse components and frequency are fixed, and \eqref{eq:wave_equation} is solved for the longitudinal propagation constant, e.g., in diffraction analyses, selfadjointness holds only in special instances. This paper studies the latter case for lamellar gratings. 

Since the periodicity is one-dimensional, we assume the transverse plane is $xy$, the permittivity profile is $\epsilon(x)$ with periodicity $\ell$, and the plane of incidence is $xz$ (e.g., see Fig.~\ref{fig:setting}). The latter assumption simplifies the problem into distinct scalar TE/TM equations, given the incident waves $\bm{E}_{inc}=E_{inc}\,\hat{\bm y}$ or $\bm{H}_{inc}=H_{inc}\,\hat{\bm y}$, respectively. Otherwise, polarization coupling cannot be avoided under conical incidence~\cite{Li1993b}. 

Considering the time dependence $e^{i\omega t}$ and longitudinal dependence $e^{-i\beta z}$, the vectorial wave equations are reduced to different second-order scalar eigenvalue equations for TE/TM polarizations. Under TE incidence, where the electric field inside the grating is $\bm{E}=E_y\,\hat{\bm y}$, one obtains the well-known eigenvalue equation in terms of $E_y$, as follows:
\begin{equation}\label{eq:TE_basic_short}
\partial_x^2 \,E_y+ \big(k_0^2\epsilon-\beta^2\big)E_y = 0.
\end{equation}
where we use $\partial_x\!=\!d/dx$ for simplicity. This can be written in SL form
\begin{equation}\label{eq:TE_SL_short}
-(pE_y')' + qE_y = \lambda\, wE_y,
\end{equation}
with
\begin{equation}\label{eq:TE_ident_short}
w=p=1,\qquad q=-k_0^2\epsilon(x),\qquad \lambda=-\beta^2.
\end{equation}
Due to the losslessness assumption, all coefficients are real, and since $w=p>0$, the associated SL operator remains selfadjoint, even for metallic gratings where $\epsilon$ changes sign in the optical regime.

For the TM case, the magnetic field becomes $\bm{H}=H_y\,\hat{\bm y}$, and the corresponding eigenvalue equation for $H_y$ takes the form
\begin{equation}\label{eq:TM_basic_short}
\partial_x\!\left(\epsilon^{-1}\partial_x H_y\right)
+ \big(k_0^2 - {\beta^2}\epsilon^{-1}\big) H_y = 0.
\end{equation}
Recast in SL form, this becomes
\begin{equation}\label{eq:TM_SL_short}
-(pH_y')' + qH_y = \lambda\, w H_y,
\end{equation}
with
\begin{equation}\label{eq:TM_ident_short}
w=p=\epsilon(x)^{-1},\qquad q=-k_0^2,\qquad \lambda=-\beta^2.
\end{equation}
Both scalar equations appear in this form in~\cite{Botten1981a, Botten1981b, Li1993b, Li2001Chap}, though in numerical modal methods they are mostly treated separately, and the identification $w=p=\epsilon(x)^{-1}$ is seldom made explicit. Continuity of $H_y$ and $\epsilon^{-1}H_y'$ across interfaces follows from the tangential boundary conditions of Maxwell's equations. Moreover, periodicity imposes the Bloch–Floquet pseudo-periodic condition
\begin{equation}\label{eq:Bloch_short}
H_y(x+\ell)=e^{-i k_{x,0} \ell}H_y(x).
\end{equation}
where $k_{x,0}$ is enforced by the incident plane wave. When $\epsilon(x)$ changes sign, as in metallic regions, the principal coefficient and the weight change sign at the same points, placing the problem outside both the definite theory and the classical indefinite one~\cite{Moller1999, Kong2003}.

A closely related SL formulation arises for TM modes when the lamellar grating is also magnetic, or more commonly in parallel-plate, metal--insulator--metal (MIM), or slab waveguides, where both $\epsilon(x)$ and $\mu(x)$ may vary spatially~\cite{Kocabas2009, Hanson2013Book}. In such problems, SL coefficients are
\begin{equation}
w=p=\epsilon(x)^{-1}, \qquad q=-k_0^2 \mu(x), \qquad \lambda=-\beta^2.
\end{equation}
For a parallel-plate waveguide with perfectly conducting walls, the tangential electric field vanishes at metallic surfaces, implying $E_z=0$, and hence, $\partial_x{H_y}\big|_{\mathrm{PEC}}\!=\!0$, which corresponds to a \emph{Neumann} boundary condition for $H_y$. At higher frequencies, the model tends to a singular SL equation due to the unbounded domain, but the underlying coefficient structure remains the same. In all of the formulations above, the principal coefficient and the weight coincide. The analysis of Sec.~\ref{sec:spectral} covers the more general situation in which they may differ by a positive factor.


\section{Spectral Analysis}
\label{sec:spectral}

In this section, we present a theorem on eigenvalue boundedness in the left half-plane for a class of piecewise smooth SL problems with pseudo-periodic boundary conditions, motivated by the TM formulation of Sec.~\ref{sec:physical}. A proposition then converts this into an explicit bound, and a corollary extends the result to Dirichlet and Neumann boundary conditions. Two remarks follow: that this boundedness implies the finiteness of eigenvalues in the left half-plane, and that the result extends to a coefficient ratio of constant phase. We begin with a summary of the notation used throughout.

\medskip
\subsection{Notation}\label{subsec:notation}

For $m$ a positive integer, we define $\mathbb{N}_m := \{n\in\mathbb{N}:n\le m\}$ where $\mathbb{N}$ is the set of positive integers. Let $\Omega\!=\![a,b] \subset \mathbb{R}$ denote a bounded interval partitioned finitely as follows:
\begin{equation}\label{eq:notation_Omegaj}
  a = x_0 < x_1 < \dots < x_m = b, \qquad \Omega_j := (x_{j-1}, x_j),
\end{equation}
where $j\in \mathbb{N}_m$. We also denote
\begin{equation}\label{eq:notation_ellj}
  \ell_j := x_j - x_{j-1}, \qquad \ell := b-a = \sum\limits_{j=1}^{m}\, \ell_j.
\end{equation}
Further, given a set $I \subset \mathbb{R}$, the closure $\overline{I}$ is the smallest closed subset of $\mathbb{R}$ containing $I$.

All functions are assumed complex-valued unless stated otherwise.
For $x \in \mathbb{R}$ and a function $f$, we denote the one-sided limits (whenever they exist) by
\begin{equation}\label{eq:one_sided_limits}
f(x^-) := \lim_{y \to x^-} f(y), \qquad
f(x^+) := \lim_{y \to x^+} f(y).
\end{equation}
As in Sec.~\ref{sec:physical}, for a function of a single variable $t$ we write $\partial_t\!=\!d/dt$. For functions of $x$, the prime denotes $\partial_x$, and for functions of $\xi$ the dot denotes $\partial_\xi$.

In this text, the notation $f=O(g)$ means that $|f|\le C|g|$ for some real constant $C>0$ independent of the asymptotic parameter under consideration. Moreover, for a matrix-valued function $H=[h_{ij}]$, we define
\begin{equation}\label{eq:notation_bigOmatrix}
O(H):=[O(h_{ij})],
\end{equation}
denoting an asymptotic class of matrix-valued functions whose entries satisfy the corresponding entrywise asymptotic bounds. Algebraic operations involving $O(H)$ are performed using standard matrix addition and multiplication, with the asymptotic bounds propagated entrywise.

The Banach space of Lebesgue-integrable functions $f$ on $I$ is denoted by $L^1(I)$. Likewise, the Banach space of essentially bounded functions $f$ on $I$ is denoted by $L^\infty(I)$ with the norm
\begin{equation}\label{eq:notation_inf_norm}
  \|f\|_{L^\infty(I)} := \operatorname*{ess\,sup}_{x \in I} |f(x)|.
\end{equation}
For an open interval $I$, the Sobolev space $W^{k,\infty}(I)$ consists of functions whose weak derivatives up to order $k$ belong to $L^\infty(I)$. Its norm is given by
\begin{equation}\label{eq:notation_Wk_inf}
\|f\|_{W^{k,\infty}(I)} := \sum_{j=0}^{k} \|f^{(j)}\|_{L^\infty(I)}.
\end{equation}

For a closed interval $\overline{I}$, the space of $k$-times continuously differentiable functions equipped with the supremum norm above is denoted by $C^k(\overline{I})$, with $C(\overline{I})\!:=\!C^0(\overline{I})$.
A function $f$ on $\overline{I}$ is called \emph{Lipschitz} if $|f(x)-f(y)|\le C|x-y|$ for all $x,y\in\overline{I}$ and some constant $C>0$.
Finally, $AC(\overline{I})$ denotes the space of absolutely continuous functions on $\overline{I}$, namely those $f$ for which there exists $g \in L^1(I)$ with $f(x) = f(c) + \int_c^x g(t)\,dt$ for all $x,c \in \overline{I}$; equivalently, $f$ is differentiable almost everywhere, $f' \in L^1(I)$, and $f$ is the integral of $f'$. Every Lipschitz function belongs to $AC(\overline{I})$; absolutely continuous functions carry sets of measure zero to sets of measure zero, and remain absolutely continuous under composition with a monotone Lipschitz function.


\medskip
\subsection{Main Theorem}\label{subsec:theorem}

Here, we consider SL eigenvalue problems of the form
\begin{equation}\label{eq:SL_form}
  -(p\,y')' + q\,y = \lambda\,w\,y, \qquad x \in \Omega,
\end{equation}
where $p$, $q$, and $w$ are functions defined on $\Omega$, periodically extended to $\mathbb{R}$, and $\lambda\in\mathbb{C}$. A solution of \eqref{eq:SL_form} is understood to be a function $y$ satisfying
\[
y,\;py' \in AC(\Omega),
\]
such that \eqref{eq:SL_form} holds almost everywhere on $\Omega$. Throughout, we denote the associated state vector by $Y:=(y,py')^T$, which is therefore continuous across each interface,
\begin{equation}\label{eq:Y_continuity}
Y(x_j^+) = Y(x_j^-),\qquad {j\in \mathbb{N}_m}.
\end{equation}

\begin{theorem}[Left half-plane boundedness of eigenvalues]\label{theorem:main}

Consider the eigenvalue problem \eqref{eq:SL_form}.

\medskip
\noindent
Assume:

\begin{itemize}

\item Regularity of $p$ and $w$:
\begin{equation}\label{eq:pw_regularity}
p,w \in W^{2,\infty}(\Omega_j),\qquad \inf_{x \in \overline{\Omega_j}} |p(x)| > 0, \qquad \inf_{x \in \overline{\Omega_j}} |w(x)| > 0, \qquad {j\in \mathbb{N}_m}.
\end{equation}

\item Positivity of the coefficient ratio:
\begin{equation}\label{eq:ratio_positive}
\frac{w}{p} = r^2 \quad\text{for some}\quad r:\Omega\to\mathbb{R}^+ .
\end{equation}

\item Boundedness of $q$:
\begin{equation}\label{eq:q_boundedness}
q \in L^\infty(\Omega).
\end{equation}

\item Non-degenerate interface condition:
\begin{equation}\label{eq:pr_interface_condition}
(p\,r)(x_j^+) + (p\,r)(x_j^-) \neq 0,
\qquad j\in \mathbb{N}_m.
\end{equation}

\item Pseudo-periodic boundary condition:
\begin{equation}\label{eq:Y_pp_BC}
Y(b^+) = e^{i\varphi} Y(a^+),\qquad {\varphi \in \mathbb{R}}.
\end{equation}

\end{itemize}

\medskip
\noindent
Then the eigenvalues of \eqref{eq:SL_form}, that is, those $\lambda$ for which a nontrivial solution satisfying \eqref{eq:Y_pp_BC} exists, form a bounded set in the open left half-plane, i.e.,
\begin{equation}
\exists\, R_0>0 \quad : \quad\operatorname{Re}(\lambda)<0 \quad\Longrightarrow\quad |\lambda| < R_0.
\end{equation}
\end{theorem}

\medskip
\begin{proof}

The proof consists of the following steps. First, the eigenvalue problem is transformed into a Schrödinger-type form on a rescaled interval, and the local transfer and interface jump matrices are derived. This yields the monodromy matrix and the characteristic equation. Next, the local solutions are expressed in integral form and estimated on each subinterval, which yields their asymptotic behavior at the subinterval endpoints. These asymptotics are then used to derive the behavior of the monodromy matrix and to show that the characteristic equation admits no sufficiently large solutions in the complex left half-plane.

\medskip
\noindent\textbf{Step 1: Liouville transformation and transfer-matrix formulation.}
With $r$ as in \eqref{eq:ratio_positive}, which is uniquely determined since $r>0$, we introduce, for $j\in\mathbb{N}_m$, the transformed variable and the transformed partition
\begin{equation}\label{eq:xi_def}
  \xi(x) := \int_a^x r(s)\,ds,
  \qquad
  \xi_j := \xi(x_j),
  \qquad
  \widetilde\Omega_j := (\xi_{j-1},\xi_j),
  \qquad
  \tilde\ell_j := \int_{\Omega_j} r(s)\,ds .
\end{equation}

Since $p,w\in W^{2,\infty}(\Omega_j)$ are non-vanishing, the ratio $w/p$ belongs to $W^{2,\infty}(\Omega_j)$ and is bounded away from zero, whence $r=(w/p)^{1/2}\in W^{2,\infty}(\Omega_j)$ with $\inf_{x\in\overline{\Omega_j}}r(x)>0$. In particular, $r$ is bounded above and bounded away from zero on $\overline{\Omega_j}$, and therefore $\xi$ is a strictly increasing bijection of $\Omega$ onto $\widetilde\Omega:=[0,\tilde\ell]$, where $\tilde\ell:=\sum_{j=1}^m\tilde\ell_j$. It maps $\Omega_j$ onto
$\widetilde\Omega_j$, whose length is $\xi_j-\xi_{j-1}=\tilde\ell_j$. Moreover, $\xi\in AC(\Omega)$ by \eqref{eq:xi_def}, which justifies the substitution $d\xi=r\,dx$.
Since $r$ is periodically extended along with $p$ and $w$, one has $\xi(x+\ell)=\xi(x)+\tilde\ell$, so that $\xi$ extends to a bijection of $\mathbb{R}$ and $\widetilde\Omega$ is periodically extended with period $\tilde\ell$.

On each subinterval $\Omega_j$ with $j\in\mathbb{N}_m$, we apply the \emph{Liouville transformation} \cite{Birkhoff1989Book, Everitt1982} by defining
\begin{equation}\label{eq:u_def}
  u := \sigma y,
  \qquad
  \sigma := (pw)^{\frac14} = (pr)^{\frac12},
\end{equation}
regarded as a function of $\xi$. Here, on each subinterval $\Omega_j$, a branch of the square root is fixed and extended periodically, such that $\sigma(x)$ is $\ell$-periodic.
The transformed solution $u$ is then understood in the piecewise sense
\begin{equation}\label{eq:u_expression}
u(\xi) := \sum\limits_{j=1}^{m}\; \bm{1}_{\{\xi\in\widetilde\Omega_j\}}\, u_j(\xi),
\qquad u_j := u|_{\widetilde\Omega_j},
\end{equation}
where $\bm{1}_{\{\xi\in\widetilde\Omega_j\}}$ denotes the indicator function of $\widetilde\Omega_j$. Rewriting the SL equation \eqref{eq:SL_form} in terms of $u$ reduces it on each subinterval $\widetilde\Omega_j$ to the local spectral problem
\begin{equation}\label{eq:uj_gov_eq}
  \mathscr{L}_j u_j = \lambda u_j, \qquad \xi \in \widetilde\Omega_j,
\end{equation}
associated with the Schrödinger-type operator
\begin{equation}\label{eq:Lj_def}
  \mathscr{L}_j := -\partial_\xi^2 + v_j,
\end{equation}
where the transformed coefficient is given, in terms of the original variable, by
\begin{equation}\label{eq:Vj_def}
  v_j := \frac{1}{r^2}\left[\frac{q}{p}
  + \frac12\Big(\frac{p''}{p}\Big)
  - \frac14\Big(\frac{p'}{p}\Big)^2
  + \frac12\Big(\frac{r''}{r}\Big)
  - \frac34\Big(\frac{r'}{r}\Big)^2\right],
  \qquad x \in \Omega_j.
\end{equation}
Since $p,r\in W^{2,\infty}(\Omega_j)$ are bounded away from zero, we have $p',p'',r',r'',p^{-1},r^{-1}\in L^\infty(\Omega_j)$. Consequently, each term in \eqref{eq:Vj_def} belongs to $L^\infty(\Omega_j)$, and hence $v_j\in L^\infty(\Omega_j)$. Since $\xi\in AC(\Omega)$, it carries the exceptional set of this essential bound to a set of measure zero in $\widetilde\Omega_j$, so that $v_j\circ\xi^{-1}\in L^\infty(\widetilde\Omega_j)$ with the same essential supremum. We write $v_j$ for either function according to context.

Moreover, $\sigma,\sigma',\sigma^{-1}\in AC(\overline{\Omega_j})$, and absolutely continuous functions are closed under multiplication, so $u=\sigma\,y$ and $\dot{u} = \sigma^{-1}( py'+ \sigma'r^{-1}u)$ are absolutely continuous in $x$. 
Since $\xi(x_2)-\xi(x_1)\ge(\inf_\Omega r)(x_2-x_1)$ for $x_1<x_2$, the inverse satisfies
\begin{equation}\label{eq:xi_inv_lip}
\big|\xi^{-1}(\eta_2)-\xi^{-1}(\eta_1)\big| \le \frac{|\eta_2-\eta_1|}{\inf_\Omega r},
\qquad \eta_1,\eta_2\in\widetilde\Omega,
\end{equation}
that is, $\xi^{-1}$ is Lipschitz. Consequently, composition with $\xi^{-1}$, which is monotone and Lipschitz by \eqref{eq:xi_inv_lip}, leaves $u$ and $\dot{u}$ absolutely continuous in $\xi$. Since $\sigma$ is non-vanishing, the map $y\mapsto u$ is a bijection between solutions of \eqref{eq:SL_form} on $\Omega_j$ and functions $u$ with $u,\dot{u}\in AC(\overline{\widetilde\Omega_j})$ satisfying \eqref{eq:uj_gov_eq} almost everywhere.

One can express \eqref{eq:uj_gov_eq} as an equivalent first-order system
\begin{equation}\label{eq:first_order_system}
  \dot{U}_j = A_j\,U_j, \qquad \xi \in \widetilde\Omega_j,
\end{equation}
where $U_j:=(u_j,\dot{u}_j)^\top$ is the Liouville state vector on the $j$-th subinterval and  
\begin{equation}\label{eq:Aj_def}
  A_j :=
  \begin{pmatrix}
    0 & 1 \\[2pt]
    v_j-\lambda & 0
  \end{pmatrix}.
\end{equation}
The relationship between the original and the transformed state vectors $Y$ and $U_j$ can be expressed, at corresponding points $x\in\Omega_j$ and $\xi=\xi(x)$, as
\begin{equation}\label{eq:R_def}
  Y(x) = R(x)\, U_j(\xi(x)), \qquad
  R :=
  \begin{pmatrix}
    \sigma^{-1} & 0 \\[1pt]
    -\sigma'/r & \sigma
  \end{pmatrix}.
\end{equation}
Here and in what follows, all quantities involving $p$, $r$ and their derivatives at interface points are understood in the sense of one-sided limits, which exist since $p,r\in W^{2,\infty}(\Omega_j)$ implies $p,p',r,r'\in C(\overline{\Omega_j})$. Hence, $R(x)$ is well-defined on each subinterval, and is $\ell$-periodic by the periodic extension of its entries. Subsequently, one can also define the global Liouville state vector as
\begin{equation}\label{eq:U_expression}
U(\xi) := \sum\limits_{j=1}^{m}\;  \bm{1}_{\{\xi\in \widetilde\Omega_j\}}\, U_j(\xi),
\end{equation}
so that $U\big|_{\widetilde\Omega_j} = U_j$.

The fundamental matrix of \eqref{eq:first_order_system} on $\widetilde\Omega_j$ is defined as the unique solution of
\begin{equation}\label{eq:fundamental_matrix}
  \dot{\Phi}_j = A_j\,\Phi_j, \qquad \Phi_j(\xi_{j-1}) = I,
\end{equation}
where $I$ is the $2\times2$ identity matrix. Consequently, if we denote $\Phi_j=[U_{j,+}|U_{j,-}]$, then $U_{j,\pm}:=(u_{j,\pm},\dot{u}_{j,\pm})^\top$ are the two independent solutions of \eqref{eq:first_order_system}, satisfying $U_{j,+}(\xi_{j-1})=(1,0)^\top$ and $U_{j,-}(\xi_{j-1})=(0,1)^\top$. Moreover, the local transfer matrix can be defined as:
\begin{equation}\label{eq:transfer_matrix}
  T_j := \Phi_j(\xi_j)=
  \begin{pmatrix}
    u_{j,+}(\xi_j) & u_{j,-}(\xi_j) \\[1pt]
    \dot{u}_{j,+}(\xi_j) & \dot{u}_{j,-}(\xi_j)
  \end{pmatrix},
\end{equation}
which propagates the Liouville state vector across the $j$-th subinterval, i.e.,
\begin{equation}\label{eq:transfer_relation}
  U(\xi_j^-) = T_j \, U(\xi_{j-1}^+).
\end{equation}
By Liouville’s formula for linear systems (see, e.g.,~\cite{Chicone2006Book})
\begin{equation}\label{eq:liouville_formula}
\det \Phi_j(\xi)
= \det \Phi_j(\xi_{j-1})\,
\exp\!\Big({\int\nolimits_{\xi_{j-1}}^{\xi} \!\operatorname{tr} A_j(\tau)\, d\tau}\Big),
\end{equation}
and since $\Phi_j(\xi_{j-1}) = I$ and $\operatorname{tr} A_j = 0$, it follows that:
\begin{equation}\label{eq:T_det1}
  \det T_j = \det \Phi_j(\xi_j) = 1.
\end{equation}

Now, the jump relationship for the Liouville state vector $U$ at each interface $\xi=\xi_j$ can be expressed using the continuity of $Y$ from \eqref{eq:Y_continuity} and the relationship \eqref{eq:R_def}
\begin{equation}\label{eq:U_J}
  U(\xi_j^+) = J_j \, U(\xi_j^-), 
  \qquad 
  J_j := R(x_j^+)^{-1} R(x_j^-),
\end{equation}
the one-sided limits corresponding under the bijection $\xi$. For $j=m$, periodic extension of $R$ gives $R(x_m^+)=R(x_0^+)$, so that $J_m$ is well defined. Hence, the interface at $x_0=a$ is identified with that at $x_m=b$, and no separate condition arises at $j=0$. Evaluating $J_j$ explicitly gives
\begin{equation}\label{eq:J_explicit}
  J_j =
  \begin{pmatrix}
    \alpha_j & 0 \\[0.5em]
    \beta_j & \alpha_j^{-1}
  \end{pmatrix},
\end{equation}
where
\begin{equation}\label{eq:alpha_beta_def}
\alpha_j := \frac{\sigma(x_j^+)}{\sigma(x_j^-)}, 
\qquad
\beta_j := \frac{1}{2}\,
\frac{\dfrac{(p\,r)'(x_j^+)}{r(x_j^+)} - \dfrac{(p\,r)'(x_j^-)}{r(x_j^-)}}
{\sigma(x_j^+)\, \sigma(x_j^-)}.
\end{equation}
Since $R(x_j^\pm)$ are well-defined, $J_j$ is well-defined, with $\det J_j = 1$. 

Composing the transfer within each subinterval, via $T_j$, and the transmission across interfaces, via $J_j$, the global monodromy matrix for the periodic problem can be written as
\begin{equation}\label{eq:monodromy}
  {M} = \prod\limits_{j=0}^{m-1} J_{m-j} T_{m-j} =  J_m T_m \, J_{m-1} T_{m-1} \cdots J_1 T_1,
\end{equation}
which satisfies
\begin{equation}\label{eq:monodromy_relation}
  U(\tilde\ell^{\,+}) = {M}\, U(0^+).
\end{equation}
The pseudo-periodic boundary condition \eqref{eq:Y_pp_BC} implies for the Liouville state vector that $\psi\,R(a^+) U(0^+) = R(b^+) U(\tilde\ell^{\,+})$ with $\psi := e^{i\varphi}$. Since $\sigma$ is $\ell$-periodic by \eqref{eq:u_def}, one has $R(a^+) = R(b^+)$, and the boundary condition reduces to
\begin{equation}\label{eq:semi_periodic_bc}
  \psi\, U(0^+) = U(\tilde\ell^{\,+}).
\end{equation}
Comparing \eqref{eq:monodromy_relation} and \eqref{eq:semi_periodic_bc}, if we define the characteristic matrix $\Delta\!:=\! \psi I -{M}$ for the periodic problem, the eigenvalues are the zeros of
\begin{equation}\label{eq:char_det}
  \det{\Delta}
  = \psi^2 - \psi\, \operatorname{tr}{M} + \det{M}.
\end{equation}
Since each local transfer matrix and each jump matrix has determinant $1$, we have $\det{M}=1$, and therefore the eigenvalues satisfy
\begin{equation}\label{eq:D_final}
  \operatorname{tr}{M} = 2\cos(\varphi).
\end{equation}

\medskip
\noindent\textbf{Step 2: Local solutions and uniform errors.} 
Let $\kappa := (-\lambda)^{\frac12} = \rho + i\nu$ with $\rho, \nu \in \mathbb{R}$ and the branch chosen so that $\rho>0$. Knowing that $\operatorname{Re}(\lambda)\!=\!\nu^2\!-\!\rho^2\!<\!0$, what we need to show is that $|\kappa|$, equivalently $\rho$, cannot tend to $+\infty$, while $|\nu|<\rho$. Since the latter inequality implies $\rho<|\kappa|<\sqrt{2}\,\rho$, henceforth, we may use the following equivalence:
\begin{equation}\label{eq:O_equivalence}
  \kappa^d\,e^{c\kappa}\,O\big(h(\rho)\big)=O\big(\kappa^d\,e^{c\kappa}\,h(\rho)\big)=O\big(\rho^d\,e^{c\rho}\,h(\rho)\big)=\rho^d\,e^{c\rho}\,O\big(h(\rho)\big),
\end{equation}
for constants $c,d\in\mathbb{R}$ and non-negative function $h$.

We begin by rewriting \eqref{eq:uj_gov_eq} in the following form:
\begin{equation}\label{eq:uj_eq}
  \mathscr{D}_\kappa u_j = v_j u_j,\qquad \xi\in \widetilde\Omega_j,
\end{equation}
where $\mathscr{D}_\kappa := \partial_\xi^2 - \kappa^2$ is a differential operator with constant coefficients. For $\xi,\tau\in\widetilde\Omega_j$, we denote by $G(\xi,\tau)$ the Green function satisfying
\begin{equation}\label{eq:G_eq}
  \mathscr{D}_\kappa G(\xi,\tau) = \delta(\xi-\tau), \qquad 
  G(\xi_{j-1},\tau) = \partial_\xi G(\xi_{j-1},\tau) = 0.
\end{equation}
The normalization conditions at $\xi_{j-1}$ ensure uniqueness of $G(\cdot,\tau)$ and correspond to the choice of fundamental solutions normalized at the left endpoint. Its explicit form is
\begin{equation}\label{eq:G_explicit_indicator}
G(\xi,\tau) = \kappa^{-1} \sinh\!\big(\kappa(\xi-\tau)\big) \, \bm{1}_{\{\tau \le \xi\}},
\end{equation}
where $\bm{1}_{\{\tau\le \xi\}}$ denotes the indicator of the set
$\{(\xi,\tau)\in\widetilde\Omega_j\times\widetilde\Omega_j:\ \tau\le \xi\}$.

Applying the Green function method \cite{Dudley2015Book} to \eqref{eq:uj_eq}, the solution can be expressed explicitly as
\begin{equation}\label{eq:G_explicit}
u_j(\xi)
= \int_{\widetilde\Omega_j} G(\xi,\tau)\, v_j(\tau)\, u_j(\tau)\, d\tau
- \Big[\, \dot{u}_j(\tau) \, G(\xi,\tau) - u_j(\tau) \, \partial_\tau G(\xi,\tau) \, \Big]_{\tau=\xi_{j-1}}^{\tau=\xi_j}.
\end{equation}
Here, the bracketed term accounts for the contributions from the endpoints. At the upper endpoint $\tau=\xi_j$, the boundary contribution vanishes since $G(\xi,\tau)$ and $\partial_\tau G(\xi,\tau)$ are supported on $\{\tau \le \xi\}$, while $\xi < \xi_j$ for all $\xi \in \widetilde\Omega_j$. At the lower endpoint $\tau = \xi_{j-1}$, the Green function terms take the form of the homogeneous solutions of $\mathscr{D}_\kappa$ operator over $\widetilde\Omega_j$. Specifically, the above equation can be written as
\begin{equation}\label{eq:u_Green_with_u_h}
u_j(\xi)=u_j(\xi_{j-1})\,\tilde{u}_{j,+}(\xi) + \dot{u}_j(\xi_{j-1})\,\tilde{u}_{j,-}(\xi)
+\int_{\widetilde\Omega_j} G(\xi,\tau)\,v_j(\tau)\,u_j(\tau)\,d\tau,
\end{equation}
where $\tilde{u}_{j,\pm}$ denote the homogeneous solutions of $\mathscr{D}_\kappa \tilde{u}_{j,\pm}=0$ with the normalized initial conditions, explicitly given by
\begin{equation}\label{eq:u_h_def}
\tilde{u}_{j,+}(\xi):=\cosh\!\big(\kappa(\xi-\xi_{j-1})\big),\qquad
\tilde{u}_{j,-}(\xi):={\kappa}^{-1}\sinh\!\big(\kappa(\xi-\xi_{j-1})\big),
\end{equation}
with derivatives
\begin{equation}\label{eq:u_h_prime_def}
\dot{\tilde{u}}_{j,+}(\xi)=\kappa\sinh\!\big(\kappa(\xi-\xi_{j-1})\big),\qquad
\dot{\tilde{u}}_{j,-}(\xi)=\cosh\!\big(\kappa(\xi-\xi_{j-1})\big).
\end{equation}

If we denote the two independent solutions of $\mathscr{D}_\kappa u_j=v_j u_j$ by $u_{j,\pm}$, such that
\begin{equation}\label{eq:u_ic}
u_{j,+}(\xi_{j-1})=1,\quad \dot{u}_{j,+}(\xi_{j-1})=0,\qquad
u_{j,-}(\xi_{j-1})=0,\quad \dot{u}_{j,-}(\xi_{j-1})=1,
\end{equation}
they can be expressed satisfying the following integral equations:
\begin{equation}\label{eq:u_integral}
u_{j,\pm}= \tilde{u}_{j,\pm} + \int_{\widetilde\Omega_j} G(\xi,\tau)\,v_j(\tau)\,u_{j,\pm}(\tau)\,d\tau.
\end{equation}
The integral term of \eqref{eq:u_integral} may be differentiated with respect to $\xi$ without recourse to a differentiation-under-the-integral theorem. Indeed, denoting $g_\pm := v_j\, u_{j,\pm}$, which belongs to $L^1(\widetilde\Omega_j)$ since $v_j\in L^\infty(\widetilde\Omega_j)$ and $u_{j,\pm}\in AC(\overline{\widetilde\Omega_j})$, the indicator in \eqref{eq:G_explicit_indicator} restricts the integration to $\tau\in(\xi_{j-1},\xi)$, and the addition formula for the hyperbolic sine separates the variables as
\begin{equation}\label{eq:I_split}
\int_{\widetilde\Omega_j} G(\xi,\tau)\,g_\pm(\tau)\,d\tau
=
\frac{\sinh(\kappa \xi)}{\kappa}\,\Psi^\pm_{\mathrm{c}}(\xi)
-
\frac{\cosh(\kappa \xi)}{\kappa}\,\Psi^\pm_{\mathrm{s}}(\xi),
\end{equation}
where
\begin{equation}\label{eq:Psi_def}
\Psi^\pm_{\mathrm{c}}(\xi) := \int_{\xi_{j-1}}^{\xi}\!\cosh(\kappa \tau)\,g_\pm(\tau)\,d\tau,
\qquad
\Psi^\pm_{\mathrm{s}}(\xi) := \int_{\xi_{j-1}}^{\xi}\!\sinh(\kappa \tau)\,g_\pm(\tau)\,d\tau.
\end{equation}
The integrands in \eqref{eq:Psi_def} belong to $L^1(\widetilde\Omega_j)$, so $\Psi^\pm_{\mathrm{c}},\Psi^\pm_{\mathrm{s}}\in AC(\overline{\widetilde\Omega_j})$ with $\dot{\Psi}^\pm_{\mathrm{c}}=\cosh(\kappa \xi)g_\pm$ and $\dot{\Psi}^\pm_{\mathrm{s}}=\sinh(\kappa \xi)g_\pm$ almost everywhere. Since $\sinh(\kappa \xi)$ and $\cosh(\kappa \xi)$ are continuously differentiable in $\xi$ on $\overline{\widetilde\Omega_j}$, the right-hand side of \eqref{eq:I_split} belongs to $AC(\overline{\widetilde\Omega_j})$ and the product rule applies almost everywhere. Now, we may substitute \eqref{eq:I_split} into \eqref{eq:u_integral} and differentiate both sides. Since the terms carrying $\dot{\Psi}^\pm_{\mathrm{c}}$ and $\dot{\Psi}^\pm_{\mathrm{s}}$ vanish, i.e.,
\begin{equation}\label{eq:cancel}
\frac{\sinh(\kappa \xi)}{\kappa}\,\dot{\Psi}^\pm_{\mathrm{c}}
-
\frac{\cosh(\kappa \xi)}{\kappa}\,\dot{\Psi}^\pm_{\mathrm{s}} = 0,
\end{equation}
the remaining terms recombine by the subtraction formula for the hyperbolic cosine, yielding
\begin{equation}\label{eq:up_integral}
\dot{u}_{j,\pm}= \dot{\tilde{u}}_{j,\pm} + \int_{\widetilde\Omega_j} \partial_\xi G(\xi,\tau)\,v_j(\tau)\,u_{j,\pm}(\tau)\,d\tau,
\qquad
\partial_\xi G(\xi,\tau) = \cosh\!\big(\kappa(\xi-\tau)\big)\,\bm{1}_{\{\tau\le \xi\}} .
\end{equation}
In particular, the right-hand side of \eqref{eq:up_integral} is the sum of a continuously differentiable function and an indefinite integral of an $L^1$ function, so $\dot{u}_{j,\pm}\in AC(\overline{\widetilde\Omega_j})$, consistently with the solution class fixed in Sec.~\ref{subsec:theorem}.

We intend to estimate the local solutions uniformly, and for that, we start by bounding the homogeneous solutions and the Green function. Note that using the definition of the hyperbolic functions for $\kappa\in\mathbb{C}$, $\rho:=\operatorname{Re}(\kappa)>0$, and $t\in\mathbb{R}^+$, one readily verifies that
\begin{equation}\label{eq:hyp_bounds1}
\sinh(\rho t)
\le
|\sinh(\kappa t)|,|\cosh(\kappa t)|
\le
\cosh(\rho t)
\le
e^{\rho t}.
\end{equation}
Applying \eqref{eq:hyp_bounds1} with $t=\xi-\xi_{j-1}$ to the homogeneous solutions \eqref{eq:u_h_def}, and with $t=\xi-\tau$ to the Green function \eqref{eq:G_explicit_indicator} and its derivative in \eqref{eq:up_integral}, one obtains for $\xi_{j-1} < \tau \le \xi \le \xi_j$
\begin{alignat}{2}
  |\tilde{u}_{j,+}| & \le e^{\rho (\xi-\xi_{j-1})},
  \qquad &
  |\tilde{u}_{j,-}| & \le |\kappa|^{-1} e^{\rho (\xi-\xi_{j-1})},
  \label{eq:uh_bounds}
  \\
  |G| & \le |\kappa|^{-1} e^{\rho (\xi-\tau)},
  \qquad &
  |\partial_\xi G| & \le e^{\rho (\xi-\tau)} .
  \label{eq:G_bounds}
\end{alignat}

Let us denote $V_j := \|v_j\|_{L^\infty(\widetilde\Omega_j)}$. To compensate for the exponential kernel bounds in \eqref{eq:G_bounds}, we measure the local solutions against the same exponential growth, defining
\begin{equation}\label{eq:Mj_def}
M_{j,\pm} := \sup_{\xi\in\overline{\widetilde\Omega_j}}\, |u_{j,\pm}(\xi)|\, e^{-\rho (\xi-\xi_{j-1})},
\end{equation}
which is finite since $u_{j,\pm}\in AC(\overline{\widetilde\Omega_j})$. With this choice the two exponentials cancel inside the integrals of both \eqref{eq:u_integral} and \eqref{eq:up_integral}, since $|u_{j,\pm}(\tau)|\le M_{j,\pm}\,e^{\rho (\tau-\xi_{j-1})}$ together with \eqref{eq:G_bounds} gives
\begin{alignat}{2}
  |G|\,|u_{j,\pm}(\tau)|
  & \le
  \frac{M_{j,\pm}}{|\kappa|}\; e^{\rho (\xi-\tau)}\, e^{\rho (\tau-\xi_{j-1})}
  =
  \frac{M_{j,\pm}}{|\kappa|}\; e^{\rho (\xi-\xi_{j-1})},
  \label{eq:kernel_cancel}
  \\
  |\partial_\xi G|\,|u_{j,\pm}(\tau)|
  & \le
  M_{j,\pm}\; e^{\rho (\xi-\tau)}\, e^{\rho (\tau-\xi_{j-1})}
  =
  M_{j,\pm}\; e^{\rho (\xi-\xi_{j-1})},
  \label{eq:kernel_cancel_deriv}
\end{alignat}
both independent of the integration variable. The first of these suffices to bound $M_{j,\pm}$ itself, since \eqref{eq:up_integral} expresses $\dot{u}_{j,\pm}$ through $u_{j,\pm}$ and requires no separate a priori estimate. Taking absolute values pointwise in \eqref{eq:u_integral} and using \eqref{eq:uh_bounds} and \eqref{eq:kernel_cancel} therefore yields, for every $\xi\in\overline{\widetilde\Omega_j}$,
\begin{align}
|u_{j,\pm}(\xi)|
&\;\le\;
|\tilde{u}_{j,\pm}(\xi)|
+
\int_{\xi_{j-1}}^{\xi}\! \big|G(\xi,\tau)\big|\,\big|v_j(\tau)\big|\,\big|u_{j,\pm}(\tau)\big|\,d\tau
\nonumber
\\
&\;\le\;
|\tilde{u}_{j,\pm}(\xi)|
+
\frac{V_j\,M_{j,\pm}}{|\kappa|}\; e^{\rho (\xi-\xi_{j-1})}\!\!\int_{\xi_{j-1}}^{\xi}\!\! d\tau
\nonumber
\\
&\;\le\;
|\tilde{u}_{j,\pm}(\xi)|
+
r_j\, M_{j,\pm}\, e^{\rho (\xi-\xi_{j-1})},
\qquad
r_j:=\frac{\tilde\ell_j\, V_j}{|\kappa|} = O(\kappa^{-1}).
\label{eq:pointwise_chain}
\end{align}
Multiplying \eqref{eq:pointwise_chain} by $e^{-\rho (\xi-\xi_{j-1})}$, bounding the homogeneous terms by \eqref{eq:uh_bounds}, and taking the supremum over $\xi\in\overline{\widetilde\Omega_j}$ as in \eqref{eq:Mj_def} gives
\begin{equation}\label{eq:M_selfbound}
M_{j,+}\le 1+r_j\,M_{j,+},
\qquad
M_{j,-}\le |\kappa|^{-1}+r_j\,M_{j,-} .
\end{equation}
We would like to ensure that $\rho$ is large enough so that $r_j\le\tfrac12$ for all $j\in\mathbb{N}_m$, where $\tilde\ell_j$ is given by \eqref{eq:xi_def}. Since $\rho<|\kappa|$, this is guaranteed whenever
\begin{equation}\label{eq:rho_min}
\rho \ge \rho_{0} := 2\max\limits_{j\in\mathbb{N}_m} \big\{\tilde\ell_j\, V_j\big\},
\end{equation}
in which case \eqref{eq:M_selfbound} may be solved for $M_{j,\pm}$ to give
\begin{equation}\label{eq:M_bound}
M_{j,+}\le 2,
\qquad
M_{j,-}\le 2\,|\kappa|^{-1}.
\end{equation}

Now, one may rewrite \eqref{eq:u_integral} and \eqref{eq:up_integral} as follows:
\begin{equation}\label{eq:deviation_integrals}
u_{j,\pm}-\tilde{u}_{j,\pm}
=
\int_{\xi_{j-1}}^{\xi}\! G\,v_j\,u_{j,\pm}\,d\tau,
\qquad
\dot{u}_{j,\pm}-\dot{\tilde{u}}_{j,\pm}
=
\int_{\xi_{j-1}}^{\xi}\! \partial_\xi G\,v_j\,u_{j,\pm}\,d\tau .
\end{equation}
Taking absolute values in \eqref{eq:deviation_integrals}, applying the bounds \eqref{eq:kernel_cancel}--\eqref{eq:kernel_cancel_deriv} with \eqref{eq:M_bound} and $|v_j|\le V_j$, and integrating over an interval of length at most $\tilde\ell_j$ therefore yields the following uniform errors, for all $\xi\in\overline{\widetilde\Omega_j}$:
\begin{alignat}{2}
  \big|u_{j,+}-\tilde{u}_{j,+}\big| & \le 2\,r_j\, e^{\rho (\xi-\xi_{j-1})},
  \qquad &
  \big|\dot{u}_{j,+}-\dot{\tilde{u}}_{j,+}\big| & \le 2\,|\kappa|\,r_j\, e^{\rho (\xi-\xi_{j-1})},
  \label{eq:u_plus_error}
  \\
  \big|u_{j,-}-\tilde{u}_{j,-}\big| & \le 2\,|\kappa|^{-1} r_j\, e^{\rho (\xi-\xi_{j-1})},
  \qquad &
  \big|\dot{u}_{j,-}-\dot{\tilde{u}}_{j,-}\big| & \le 2\,r_j\, e^{\rho (\xi-\xi_{j-1})}.
  \label{eq:u_minus_error}
\end{alignat}

\medskip
\noindent\textbf{Step 3: Asymptotic monodromy analysis.}
Evaluating \eqref{eq:u_plus_error}--\eqref{eq:u_minus_error} at $\xi=\xi_j$, and using equivalence relationships \eqref{eq:O_equivalence} and $r_j=O(\kappa^{-1})$, the end point traces $u_{j,\pm}$ and $\dot{u}_{j,\pm}$, and therefore, the transfer matrix $T_j$, defined in \eqref{eq:transfer_matrix}, yields
\begin{equation}\label{eq:transfer_matrix_error}
  T_j-\tilde{T}_j =
  \frac12\,\kappa^{-1}\, e^{\kappa\tilde{\ell}_j}
  \begin{pmatrix}
    O(1) & O(\kappa^{-1})
    \\[1pt]
    O(\kappa^{-1}) & O(1)
  \end{pmatrix},
\end{equation}
wherein, we have defined
\begin{equation}\label{eq:transfer_matrix_tilde}
  \tilde{T}_j :=
  \begin{pmatrix}
    \tilde{u}_{j,+}(\xi_j) & \tilde{u}_{j,-}(\xi_j) \\[1pt]
    \dot{\tilde{u}}_{j,+}(\xi_j) & \dot{\tilde{u}}_{j,-}(\xi_j)
  \end{pmatrix}
  =
  \frac{1}{2}e^{\kappa\tilde{\ell}_j}
  \begin{pmatrix}
    1+e^{-2\kappa\tilde{\ell}_j}
      & \kappa^{-1}\big(1-e^{-2\kappa\tilde{\ell}_j}\big)
    \\[3pt]
    \kappa\big(1-e^{-2\kappa\tilde{\ell}_j}\big)
      & 1+e^{-2\kappa\tilde{\ell}_j}
  \end{pmatrix},
\end{equation}
elaborated according to \eqref{eq:u_h_def} and \eqref{eq:u_h_prime_def}. Note that the constant coefficients of the asymptotic entries in \eqref{eq:transfer_matrix_error} are arbitrary by definition.
Now, using \eqref{eq:transfer_matrix_tilde} in \eqref{eq:transfer_matrix_error}, and the fact that $e^{-2\kappa\tilde\ell_j}=O(\kappa^{-1})$, one immediately obtains the following asymptotic relationship for the transfer matrix:
\begin{equation}\label{eq:transfer_matrix_asymp}
  T_j =
  \frac12 e^{\kappa\tilde{\ell}_j}
  \begin{pmatrix}
    1+O(\kappa^{-1}) &
    \kappa^{-1}\bigl(1+O(\kappa^{-1})\bigr)
    \\[3pt]
    \kappa\bigl(1+O(\kappa^{-1})\bigr) &
    1+O(\kappa^{-1})
  \end{pmatrix},
\end{equation}
which can be more compactly written as
\begin{equation}\label{eq:Tj_asymp}
  T_j = e^{\kappa \tilde\ell_j} \Big(S + \kappa^{-1}O(S)\Big),\qquad
    S :=   
  \frac12 \begin{pmatrix}
    1 & \kappa^{-1}\\[1pt]
    \kappa & 1
  \end{pmatrix}.
\end{equation}
Note that, the asymptotic matrix class $O(S)$, is closed under addition and multiplication, i.e., if $M_1=O(S)$ and $M_2=O(S)$ are two matrices in this class, then $M_1+M_2=O(S)$, $M_1\, M_2=O(S)$, and $M_2\, M_1=O(S)$. Moreover, if $g(\kappa)=O(\kappa^c)$ is a scalar function for $c\in\mathbb{R}$, then $g\, M_i=\kappa^c\,O(S)$ for $i=1,2$. We can sum up these properties as follows:
\begin{equation}\label{eq:O_products}
O(S)+O(S)=O(S),\qquad O(S)^2=O(S),\qquad O(\kappa^c)\,O(S)=\kappa^c\,O(S).
\end{equation}
Moreover, define $S_z$ for $z\in\mathbb{C}\!\setminus\!\{0\}$
\begin{equation}\label{eq:Sz_def}
 S_z := \operatorname{diag}(z,z^{-1})S,
\end{equation}
noting that $S_1=S$ and $S_z=O(S)$.
Also note for $z_1,z_2\in\mathbb{C}\!\setminus\!\{0\}$
\begin{equation}\label{eq:S_products}
  S_{z_1}\,S_{z_2} = \gamma_{z_2}\,S_{z_1},\qquad \gamma_z := \operatorname{tr}\,S_z=\frac12(z + z^{-1}).
\end{equation}
Given the above properties and recalling that $\alpha_j\neq0$ in \eqref{eq:alpha_beta_def}, we evaluate the product of $J_j$ from \eqref{eq:J_explicit}, by $T_j$ from \eqref{eq:Tj_asymp}. Noting that $J_j\,S=S_{\alpha_j}+\kappa^{-1}O(S)$ and $J_j\,O(S)=O(S)$, one obtains 
\begin{equation}\label{eq:JT_mat}
  J_j\,T_j = e^{\kappa \tilde\ell_j} \Big(S_{\alpha_j} + \kappa^{-1}O(S)\Big).
\end{equation}

Now, we can compute the monodromy matrix of \eqref{eq:monodromy} asymptotically for $m\ge 2$ using $J_j\,T_j$ terms obtained above and \eqref{eq:O_products}--\eqref{eq:S_products} for simplification. Recalling that $\tilde\ell=\sum_{j=1}^{m}\tilde\ell_j$ and denoting $n_\varepsilon:=\sum_{j=0}^{m-1}\varepsilon_j$ for a sequence $\{\varepsilon_j\}_{j=0}^{m-1}$, one may write

\begin{align}\label{eq:M_asymp_derivation}
  e^{-\kappa\tilde\ell}{M} 
  &= \prod\limits_{j=0}^{m-1} \Big( S_{\alpha_{m-j}} + \kappa^{-1} O(S) \Big)
  = \sum\limits_{\substack{\varepsilon_j\in\{0,1\}}} 
      \;\prod\limits_{j=0}^{m-1}\big(S_{\alpha_{m-j}}\big)^{1-\varepsilon_j}\,\big(\kappa^{-1}O(S)\big)^{\varepsilon_j} \nonumber \\
  &= \prod\limits_{j=0}^{m-1} S_{\alpha_{m-j}} 
    + \sum\limits_{n_\varepsilon>0 } 
       \;\prod\limits_{j=0}^{m-1}\big(S_{\alpha_{m-j}}\big)^{1-\varepsilon_j}\,\big(\kappa^{-1}O(S)\big)^{\varepsilon_j} \nonumber \\    
  &=  \prod\limits_{j=0}^{m-1} S_{\alpha_{m-j}} 
    + \sum\limits_{n_\varepsilon>0 } 
       O(S)^{m-{n_\varepsilon}}\,\kappa^{-n_\varepsilon}\,O(S)^{n_\varepsilon} \nonumber \\    
 &=  \prod\limits_{j=0}^{m-1} S_{\alpha_{m-j}} 
    + \big(m\kappa^{-1}+O(\kappa^{-2})\big)\,O(S)^m \nonumber \\ 
  &= \Big(\prod_{j=1}^{m-1} \gamma_{\alpha_j}\Big) S_{\alpha_m} + \kappa^{-1} O(S).
\end{align}
The final result can be more compactly expressed as
\begin{equation}\label{eq:M_asymptotic}
  {M} = e^{\kappa\tilde\ell}\Big(\Gamma_{m-1}\,S_{\alpha_{m}} + \kappa^{-1} O(S)\Big),
\end{equation}
where we define
\begin{equation}\label{eq:Gamma_def}
  \Gamma_0 := 1,\qquad\Gamma_k := \prod\limits_{j=1}^k \gamma_{\alpha_j},\quad k\in\mathbb{N}_m.
\end{equation}
For $m=1$ the same expression follows directly from \eqref{eq:JT_mat}, since ${M}=J_1T_1$ and $\Gamma_0=1$.
Now, it follows immediately from \eqref{eq:M_asymptotic} that
\begin{align}\label{eq:tr_M_asymp}
  \operatorname{tr} {M}
  &= e^{\kappa\tilde\ell}\Big(\Gamma_{m-1}\,\operatorname{tr}\,S_{\alpha_{m}}+ O(\kappa^{-1})\Big) \nonumber\\
  &= e^{\kappa\tilde\ell}\Big(\Gamma_{m-1}\,\gamma_{\alpha_m} + O(\kappa^{-1})\Big)\nonumber\\
  &= e^{\kappa\tilde\ell}\Big(\Gamma_m + O(\kappa^{-1})\Big).
\end{align}

We note that $\Gamma_m \neq 0$, since $\gamma_{\alpha_j} = \tfrac12(\alpha_j + \alpha_j^{-1}) \neq 0$ for each $j\in\mathbb{N}_m$. The latter is followed by the fact that $\alpha_j\neq \pm i$ or equivalently $\sigma(x_j^+)^2+\sigma(x_j^-)^2\neq 0$, which is the non-degenerate interface condition \eqref{eq:pr_interface_condition}. Now we conclude the argument. All estimates above presuppose \eqref{eq:rho_min}, which is satisfied for $\rho$ sufficiently large. It then follows from \eqref{eq:tr_M_asymp} that for sufficiently large $\rho$ there exists $c>0$ such that
\begin{equation}\label{eq:tr_M_ineq_exp}
  |\operatorname{tr}{M}| \ge c\, e^{\rho\tilde\ell}.
\end{equation}
If $\rho$ is large enough so that $|\operatorname{tr}{M}| > 2$, then \eqref{eq:D_final} yields $|\cos(\varphi)| > 1$, which is a contradiction. Consequently, $\rho$ must be bounded, provided that $|\nu|<\rho$. Equivalently, there exists $R_0>0$ such that every eigenvalue with $\operatorname{Re}(\lambda)<0$ satisfies $|\lambda|< R_0$. This completes the proof.

\end{proof}


\medskip
\subsection{Complementary Results}\label{subsec:complementary}
We conclude this section by presenting four additional results following from Theorem~\ref{theorem:main}. Proposition~\ref{proposition:bounds} converts the qualitative statement into an explicit bound on the left half-plane eigenvalues, Corollary~\ref{corollary:non-periodic} extends the result to two other boundary conditions, Remark~\ref{remark:finiteness} establishes the finiteness of such eigenvalues in \eqref{eq:SL_form}, and Remark~\ref{remark:constant_phase} treats the case of a coefficient ratio of constant phase.

\begin{proposition}[Explicit bound for left half-plane eigenvalues]
\label{proposition:bounds}
Under the assumptions of Theorem~\ref{theorem:main}, set
\begin{equation}\label{eq:E_constants}
\begin{aligned}
  \tilde{c} &:= \max_{j\in\mathbb{N}_m}\big\{|\alpha_j|,|\alpha_j|^{-1}\big\},
  \\[4pt]
  \tilde{d} &:= \max_{j\in\mathbb{N}_m}\bigg\{
    |\beta_j|
    + 2\max\big\{|\alpha_j|,|\alpha_j|^{-1},|\beta_j|/\rho_0\big\}
      \Big(4\tilde\ell_j V_j + \tfrac{1}{\sqrt{2}\,e\,\tilde\ell_j}\Big)
  \bigg\},
\end{aligned}
\end{equation}
where $\rho_0$, $V_j$, $\tilde\ell_j$, $\tilde\ell$, $\alpha_j$, $\beta_j$, and $\Gamma_m$ are as in the proof of Theorem~\ref{theorem:main}. Then every left half-plane eigenvalue $\lambda=-\kappa^2$ satisfies
\begin{equation}\label{eq:E_kappa_bound}
  |\kappa| \le \sqrt{2}\,
  \max\bigg\{
    \rho_0,
    \quad
    \frac{2\rho_0}{|\Gamma_m|}
    \Big[\big(\tilde{c}+\tilde{d}/\rho_0\big)^{m}-\tilde{c}^{\,m}\Big],
    \quad
    \frac{1}{\tilde\ell}\ln\frac{4}{|\Gamma_m|}
  \bigg\}.
\end{equation}
\end{proposition}

\noindent
The bound \eqref{eq:E_kappa_bound} is not optimised; it serves to show that the exclusion region of Theorem~\ref{theorem:main} is explicitly computable from the coefficients and the interface data. In the case $w=p$ one has $r\equiv1$, so that $\tilde\ell_j=\ell_j$ and $\tilde\ell=\ell$.
\begin{proof}
For a constant $t_0>0$, let $E[t_0]$ denote the class of $2\times 2$ matrices
\begin{equation}\label{eq:E_def}
  E[t_0] :=
  \left\{
  \begin{pmatrix}
    e_{11} & \kappa^{-1}e_{12}\\[1pt]
    \kappa e_{21} & e_{22}
  \end{pmatrix}
  :\ \sup_{\kappa}|e_{ij}(\kappa)|<\frac{t_0}{2},\quad i,j=1,2 
  \right\}.
\end{equation}
We refer to $t_0$ as the class bound. For a given matrix $M_0$, we write $M_0=E[t_0]$ to indicate that $M_0$ belongs to this class; hence, the equality sign is understood in the sense of \emph{class membership}. For a matrix $M_0=E[t_0]$, the trace is bounded by the class bound, i.e., $|\operatorname{tr}M_0|\le t_0$. The class is monotone in the bound: if $M_0=E[t_1]$ and $t_1\le t_2$, then $M_0=E[t_2]$; that is, the class bound may be increased but not decreased. If $M_1=E[t_1]$ and $M_2=E[t_2]$, then $M_1+M_2=E[t_1+t_2]$, $M_1M_2=E[t_1t_2]$, and $M_2M_1=E[t_1t_2]$. Moreover, for a scalar $z\in\mathbb{C}$ one has $zE[t_1]=E[|z|\,t_1]$. In particular, since \eqref{eq:rho_min} assumes $|\kappa|\ge\rho\ge\rho_0$, we have $\kappa^{-n}E[t_1]=E[\rho_0^{-n}t_1]$. Given these properties, we may manipulate the class members according to the following rules:
\begin{equation}\label{eq:E_properties}
\begin{aligned}
  &E[t_1]=E[t_2] \ \ (t_1\le t_2),
  \qquad
  zE[t_1]=E[|z|\,t_1],
  \qquad
  \big|\operatorname{tr}E[t_1]\big|\le t_1,
  \\[4pt]
  &E[t_1]+E[t_2]=E[t_1+t_2],
  \qquad
  E[t_1]E[t_2]=E[t_1t_2].
\end{aligned}
\end{equation}

The idea will be to express the defined matrices $T_j$, $\tilde{T}_j$, $J_j$, $S_{\alpha_j}$ using the above estimation class. Recalling the definition of $r_j=\tilde{\ell}_jV_j/|\kappa|$ from \eqref{eq:pointwise_chain} and evaluating \eqref{eq:u_plus_error}--\eqref{eq:u_minus_error} at $\xi=\xi_j$, one may reformulate the right-hand side of \eqref{eq:transfer_matrix_error} as follows:
\begin{equation}\label{eq:E_Tj_error}
  T_j-\tilde{T}_j =
  e^{\kappa\tilde{\ell}_j}\,\kappa^{-1}\,E[4\, \tilde{\ell}_j\, V_j].
\end{equation}
Moreover, $\tilde{T}_j$ in \eqref{eq:transfer_matrix_tilde}, can be written as
\begin{equation}\label{eq:E_Tj_tilde}
  \tilde{T}_j = e^{\kappa \tilde\ell_j}\, \Big(S + \kappa^{-1}E\big[1/(\sqrt{2}\, e\, \tilde{\ell}_j)\big]\Big),
\end{equation}
where $S$ is defined in \eqref{eq:Tj_asymp}. The class bound of the residual term in \eqref{eq:E_Tj_tilde} stems from the fact that $\tfrac12\big|\kappa\,e^{-2\kappa\tilde{\ell}_j}\big| \le \tfrac{1}{\sqrt{2}}\,\rho\, e^{-2\rho\tilde{\ell}_j}\le 1/(2\sqrt{2}\,e\,\tilde{\ell}_j)$, for $\rho>0$. Now, using the combination properties described above, \eqref{eq:E_Tj_error}-\eqref{eq:E_Tj_tilde} immediately yield
\begin{equation}\label{eq:E_Tj}
  T_j = e^{\kappa \tilde\ell_j}\, \Big(S + \kappa^{-1}E\big[{a_j}\big]\Big),\qquad {a_j}:=4\,\tilde{\ell}_j\,V_j + 1/(\sqrt{2}\,e\,\tilde{\ell}_j).
\end{equation}
Now, recall from \eqref{eq:rho_min} that we assumed $|\kappa|\ge \rho \ge \rho_{0}$. Using this lower bound to match the estimation class matrix form \eqref{eq:E_def} with the definition $J_j$ in \eqref{eq:J_explicit}, one obtains
\begin{equation}\label{eq:E_J}
  J_j = E\big[{b_j}\big],\qquad {b_j}:= 2 \max\big\{|\alpha_j|,|\alpha_j|^{-1},|\beta_j|/\rho_{0}\big\}.
\end{equation}
Moreover, using $S_{\alpha_j}$ as defined in \eqref{eq:Sz_def} and appearing in \eqref{eq:JT_mat}, one can simply verify
\begin{equation}\label{eq:E_JS}
  J_j S = S_{\alpha_j} + \kappa^{-1}E\big[|\beta_j|\big].
\end{equation}
Finally, expanding $S_{\alpha_j}$ itself and comparing it with \eqref{eq:E_def} yield
\begin{equation}\label{eq:E_Sa}
  S_{\alpha_j} = E[{c_j}],\qquad {c_j}:= \max\big\{|\alpha_j|,|\alpha_j|^{-1}\big\}.
\end{equation}
Now, we use equations \eqref{eq:E_Tj}, \eqref{eq:E_J}, and \eqref{eq:E_JS} to write
\begin{align}\label{eq:E_JT}
  J_j T_j
  &= e^{\kappa \tilde\ell_j}\, \Big(J_j S + \kappa^{-1}J_j E\big[{a_j}\big]\Big) \nonumber\\
  &= e^{\kappa \tilde\ell_j}\, \Big(S_{\alpha_j} + \kappa^{-1}E\big[|\beta_j|\big] + \kappa^{-1}E\big[{b_j}\big] E\big[{a_j}\big]\Big) \nonumber\\
  &= e^{\kappa \tilde\ell_j}\, \Big(S_{\alpha_j} + \kappa^{-1}E[{d_j}] \Big),\qquad {d_j}:=|\beta_j| + {a_j}{b_j}.
\end{align}

Now, we can expand the matrix $M$ of \eqref{eq:monodromy} for $m\ge 2$ using $J_j\,T_j$ terms obtained above. Using the same notations used for \eqref{eq:M_asymp_derivation}, one may write
\begin{equation}\label{eq:E_M_expansion1}
   e^{-\kappa\tilde\ell}{M} 
  = \prod\limits_{j=0}^{m-1} \Big( S_{\alpha_{m-j}} + \kappa^{-1} E[{d_{m-j}}] \Big)
   = \prod\limits_{j=0}^{m-1} S_{\alpha_{m-j}} 
   + \sum\limits_{n_\varepsilon>0 } 
   \;\prod\limits_{j=0}^{m-1}\big(S_{\alpha_{m-j}}\big)^{1-\varepsilon_j}\,\big(\kappa^{-1}E[{d_{m-j}}]\big)^{\varepsilon_j}.
\end{equation}
Simplifying the first term on the right-hand side of \eqref{eq:E_M_expansion1} from \eqref{eq:M_asymp_derivation} and rewriting the second term using \eqref{eq:E_Sa} yields
\begin{equation}\label{eq:E_M_expansion2}
  e^{-\kappa\tilde\ell}{M}
   = \Gamma_{m-1}\,S_{\alpha_{m}}
   + \sum\limits_{ n_\varepsilon>0 } 
   \;\prod\limits_{j=0}^{m-1} E[{c_{m-j}}]^{1-\varepsilon_j}\,\big(\kappa^{-1}E[{d_{m-j}}]\big)^{\varepsilon_j}.
\end{equation}
Now, we rearrange \eqref{eq:E_M_expansion2}, note that $\tilde{c}=\max_j \{c_j\}$ and $\tilde{d}=\max_j \{d_j\}$ from \eqref{eq:E_constants}, and use the simplification rules of \eqref{eq:E_properties} together with $|\kappa|\ge\rho_0$ to write
\begin{align}\label{eq:E_M_asymp}
  e^{-\kappa\tilde\ell}{M} - \Gamma_{m-1}\,S_{\alpha_{m}}
   &= \sum\limits_{n_\varepsilon>0 } 
   \;\prod\limits_{j=0}^{m-1} E[\tilde{c}]^{1-\varepsilon_j}\,\big(\kappa^{-1}E[\tilde{d}]\big)^{\varepsilon_j} \nonumber \\
   &= \sum\limits_{n=1}^{m} \binom{m}{n} \kappa^{-n}\, E[\tilde{c}]^{m-n}\,E[\tilde{d}]^{n} \nonumber \\  
   &= \kappa^{-1}\rho_{0}\,\sum\limits_{n=1}^{m} \binom{m}{n} \rho_{0}^{-n}\, E[\tilde{c}]^{m-n}\,E[\tilde{d}]^{n} \nonumber \\ 
   &= \kappa^{-1}\, E\bigg[\rho_{0}\sum\limits_{n=1}^{m} \binom{m}{n} {\tilde{c}}\,^{m-n}\,(\tilde{d}/\rho_0)^{n}\bigg] \nonumber \\   
   &= \kappa^{-1}\, E\Big[\rho_{0}\big(\tilde{c}+\tilde{d} / \rho_{0}\big)^{m} - \rho_{0}{\tilde{c}}\,^m\Big],
\end{align}
which can be more compactly put as
\begin{equation}\label{eq:E_M_compact}
  {M} = e^{\kappa\tilde\ell}\Big(\Gamma_{m-1}\,S_{\alpha_{m}} + \kappa^{-1}\, E[\rho_{\max}] \Big),\qquad \rho_{\max}:=\rho_{0}\big(\tilde{c}+\tilde{d} / \rho_{0}\big)^{m} - \rho_{0}{\tilde{c}}\,^m.
\end{equation}
Taking the trace of both sides, noting $\operatorname{tr}(\Gamma_{m-1}S_{\alpha_m})=\Gamma_m$ and $|e^{\kappa\tilde\ell}|=e^{\rho\tilde\ell}$, and using the triangle inequality yields
\begin{equation}\label{eq:E_M_trace}
  |\operatorname{tr} M| \ge e^{\rho\tilde\ell}\Big(\big|\Gamma_{m}\big| - |\kappa|^{-1}\,\big|\operatorname{tr} E[\rho_{\max}]\big| \Big)
  \ge e^{\rho\tilde\ell}\Big(\big|\Gamma_{m}\big| - \frac{\rho_{\max}}{\rho} \Big).
\end{equation}
If $\rho$ satisfies
\begin{equation}\label{eq:E_rho_bound}
  \rho >
  \max\bigg\{
    \frac{2\rho_{\max}}{|\Gamma_m|},
    \frac{1}{\tilde{\ell}}\ln\frac{4}{|\Gamma_m|},
    \rho_0
  \bigg\},
\end{equation}
then $|\operatorname{tr}M|>2$, which contradicts \eqref{eq:D_final}. Hence $\rho$ is bounded by the right-hand side of \eqref{eq:E_rho_bound}. Since $\rho<|\kappa|<\sqrt{2}\rho$, the corresponding bound on $|\kappa|$ is given by \eqref{eq:E_kappa_bound}.

\end{proof}

\begin{corollary}[Dirichlet and Neumann boundary conditions]\label{corollary:non-periodic}
Under the assumptions of Theorem~\ref{theorem:main}, with the pseudo-periodic boundary condition \eqref{eq:Y_pp_BC} replaced by a Dirichlet or Neumann condition, the left half-plane eigenvalues remain bounded.
\end{corollary}

\begin{proof}
For classical boundary conditions such as Dirichlet or Neumann, the problem is posed on $\Omega=[a,b]$ without periodic continuation, so only the left trace at $b$ is defined, equivalently the left trace at $\xi(b)=\tilde\ell$ in the transformed variable. 
In this case the non-periodic global transfer matrix is 
\begin{equation}\label{eq:monodromy_truncated}
   \widetilde{M} := T_m \, J_{m-1} T_{m-1} \cdots J_1 T_1, \qquad  U(\tilde\ell^{\,-}) = \widetilde{M}\, U(0^+).
\end{equation}
Let the general linear boundary conditions be
\begin{equation}\label{eq:Y_BC_non-periodic}
  A\,Y(a^+) + B\,Y(b^-)=0,
\end{equation}
where $A,B\in\mathbb{C}^{2\times2}$ and $\operatorname{rank} (A| B)=2$. Using the relationship \eqref{eq:R_def}, the boundary conditions in terms of $U$ become
\begin{equation}\label{eq:U_BC_non-periodic}
A\,R(a^+)\,U(0^+) +  B\,R(b^-)\,U(\tilde\ell^{\,-})=0.
\end{equation}
If we use the facts that $\widetilde{M} = J_m^{-1} {M}$ with the extended definition $J_m = R(a^+)^{-1} R(b^-)$, and that $\det R=1$, then the characteristic equation governing the eigenvalues can be expressed from \eqref{eq:U_BC_non-periodic} in terms of the extended monodromy matrix ${M}$ which we have already studied asymptotically
\begin{equation}\label{eq:char_non-periodic}
  \det\big(\widetilde{\Delta}\big)=0,\qquad \widetilde{\Delta}:=A+B\,{M}.
\end{equation}

\medskip
\noindent\textbf{Dirichlet boundary condition.}
For this case, the endpoint boundary conditions correspond to the matrix coefficients as follows:
 \begin{equation}\label{eq:Dirichlet_AB}
    y(a)=y(b)=0
    \quad\Longleftrightarrow\quad
    A=\begin{pmatrix}1&0\\0&0\end{pmatrix},\quad 
    B=\begin{pmatrix}0&0\\1&0\end{pmatrix}.
 \end{equation}
Using the above matrix coefficients to obtain $\widetilde{\Delta}$ from \eqref{eq:char_non-periodic}, gives
 \begin{equation}\label{eq:Dirichlet_Delta}
    \widetilde{\Delta} = 
    \begin{pmatrix}
    1 & 0 \\[5pt]
    {c_0 e^{\kappa \tilde\ell}}\, \big(1 + O(\kappa^{-1}) \big) &
    {c_0 e^{\kappa \tilde\ell}}\kappa^{-1}\, \big(1 + O(\kappa^{-1}) \big)
    \end{pmatrix},
 \end{equation}
where $c_0:=\tfrac12\alpha_m\Gamma_{m-1}$. Hence, the characteristic equation becomes
 \begin{equation}\label{eq:Dirichlet_det}
|\det(\widetilde\Delta)| = |c_0| e^{\rho\tilde\ell} |\kappa|^{-1}(1+O(\kappa^{-1})),
 \end{equation}
in which the dominant term grows exponentially in size, and hence, the determinant cannot vanish as $\rho\to +\infty$. 
  
\medskip
\noindent\textbf{Neumann boundary condition.}
Similar to the Dirichlet case, the matrix coefficients are obtained from the endpoint boundary conditions
 \begin{equation}\label{eq:Neumann_AB}
    y'(a)=y'(b)=0
    \quad\Longleftrightarrow\quad
    A=\begin{pmatrix}0&1\\0&0\end{pmatrix},\quad 
    B=\begin{pmatrix}0&0\\0&1\end{pmatrix},
 \end{equation}
which yields from \eqref{eq:char_non-periodic}
 \begin{equation}\label{eq:Neumann_Delta}
    \widetilde{\Delta} = 
    \begin{pmatrix}
    0 & 1 \\[5pt]
    {c'_0 e^{\kappa \tilde\ell}} \kappa\, \big(1 + O(\kappa^{-1}) \big) & {c'_0 e^{\kappa \tilde\ell}}\,\big(1 + O(\kappa^{-1}) \big)
    \end{pmatrix},
 \end{equation}
where $c'_0:=\tfrac12\alpha_m^{-1}\Gamma_{m-1}$. Taking the determinant then gives
 \begin{equation}\label{eq:Neumann_det}
    |\det\big(\widetilde{\Delta}\big)| = |c'_0|\, e^{\rho \tilde\ell}\, |\kappa|\, \big(1 + O(\kappa^{-1})\big).
 \end{equation}
Similarly, as $\rho\to +\infty$, the characteristic determinant cannot vanish due to the existence of an exponential dominant term. This concludes a similar left half-plane boundedness result for these boundary conditions.

\end{proof}

\begin{remark}[Finiteness of eigenvalues in the left half-plane]\label{remark:finiteness}
The spectral classification of regular two-point SL boundary value problems given in \cite[Lemma 3.2.4]{Zettl2005Book} implies that the eigenvalue set is either finite or countably infinite without finite accumulation points in $\mathbb{C}$, except in the degenerate case where every complex number is an eigenvalue. In the present setting, Theorem~\ref{theorem:main} ensures that all eigenvalues with $\operatorname{Re}(\lambda)\!<\!0$  are contained in a bounded subset of $\mathbb{C}$. The degenerate case is excluded here, since \eqref{eq:tr_M_ineq_exp} shows that $\operatorname{tr}{M}-2\cos(\varphi)$ is unbounded and therefore not identically zero. Since no finite accumulation points occur, it follows that only finitely many eigenvalues can lie in the open left half-plane.
\end{remark}

\begin{remark}[Coefficient ratio of constant phase]\label{remark:constant_phase}
Suppose that \eqref{eq:ratio_positive} is replaced by the more general requirement that $w/p$ have constant argument on $\Omega$, that is,
\begin{equation}\label{eq:const_phase}
\exists\, \theta\in\mathbb{R},\ \exists\, r>0 \quad : \quad \frac{w(x)}{p(x)} = e^{i\theta} r(x)^2, \qquad x\in \Omega .
\end{equation}
Writing $\hat w := e^{-i\theta} w$ and $\hat\lambda := e^{i\theta}\lambda$, the equation \eqref{eq:SL_form} becomes $-(p\,y')' + q\,y = \hat\lambda\,\hat w\,y$ with $\hat w/p = r^2$, so that Theorem~\ref{theorem:main} applies to $\hat\lambda$. Consequently the eigenvalues of \eqref{eq:SL_form} are bounded in the half-plane
\begin{equation}\label{eq:rotated_halfplane}
\operatorname{Re}\big(e^{i\theta}\lambda\big) < 0 ,
\end{equation}
which reduces to the open left half-plane when $\theta=0$. The positivity hypothesis \eqref{eq:ratio_positive} thus fixes which half-plane the conclusion concerns, rather than restricting the class of admissible coefficients. A phase varying with $x$, by contrast, cannot be removed by such a rescaling and lies outside the scope of the present analysis.
\end{remark}


\section{Numerical Illustration}
\label{sec:numerics}

We illustrate the Fourier--modal discretization of the TE and TM equations introduced in Section~\ref{sec:physical}. Throughout this section the spatial domain $\Omega$ denotes the unit cell of the grating, of length $\ell$, over which all periodic coefficients and Bloch–Floquet solutions are defined. 

\medskip
\subsection{Fourier Representation}\label{subsec:Fourier_representation}
For a pseudo-periodic field of the form
\begin{equation}
  f(x,z) = \tilde f(x)\, e^{-i\beta z},
\end{equation}
where $\tilde f$ satisfies the Bloch condition with phase $k_{x,0}$, we expand
\begin{equation}
  \tilde f(x)
  = \sum_{n=-\infty}^{\infty} f_n\, e^{- i k_{x,n} x},
  \qquad
  k_{x,n} := k_{x,0} + \frac{2\pi n}{\ell}.
\end{equation}
For a given truncation order~$N$, the associated vector of Fourier coefficients is defined by
\begin{equation}\label{eq:bracket-def}
  [f]_n := f_{\,n-N-1},
  \qquad n=1,\dots,2N+1.
\end{equation}
In the TE and TM formulations, the field components $E_y$ and $H_y$ are of this form.

If $g(x)$ is periodic on~$\Omega$, its Fourier series and coefficients are given by
\begin{equation}
  g(x)
  = \sum_{n=-\infty}^{\infty} g_n\, e^{- i n x (2\pi/\ell)},
  \qquad
  g_n = \frac{1}{\ell} \int_{\Omega}
        g(x)\, e^{\, i n x (2\pi/\ell)}\, dx,
\end{equation}
with $g=\epsilon(x)$ or $g=\epsilon(x)^{-1}$ in our applications.

\medskip
\subsection{Fourier Operators}\label{subsec:Fourier_operator}
To convert the differential equations into algebraic ones, we express differentiation and multiplication by periodic coefficients in terms of matrix operators acting on Fourier coefficient vectors.  
We begin by defining the diagonal matrix $K$ as
\begin{equation}
  K_{nn'} = \frac{k_{x,n}}{k_0}\,\delta_{nn'},
  \qquad
  k_0 = \frac{2\pi}{\lambda_{\text{vac}}}.
\end{equation}
From the Fourier representation and the bracket notation introduced in \eqref{eq:bracket-def}, differentiation of a pseudo-periodic function satisfies
\begin{equation}
  \Big[\,\frac{d\tilde f}{dx}\,\Big]
  = -\, i k_0\, K\, [f].
\end{equation}

Multiplication by a periodic coefficient $g(x)$ is represented by the Toeplitz matrix
\begin{equation}
  \llbracket g \rrbracket_{nn'} = g_{\,n-n'},
  \qquad n,n' = 1,\dots,2N+1,
\end{equation}
which yields the algebraic identity
\begin{equation}
  [\,g \tilde f\,] = \llbracket g \rrbracket\, [f].
\end{equation}

\medskip
\subsection{Modal Matrices for TE and TM Equations}\label{subsec:modal_matrices}
The modal matrices follow directly from the physical equations in Section~\ref{sec:physical}. For TE polarization, the governing scalar equation \eqref{eq:TE_basic_short} yields the Fourier–modal eigenvalue problem
\begin{equation}\label{eq:TE_modal}
  \big(\llbracket\epsilon\rrbracket - K^2\big)\,[E_y]
  = \eta^{\,2}\, [E_y],
  \qquad \eta^{\,2} := (\beta/k_0)^2,
\end{equation}
with $\eta$ denoting the normalized longitudinal propagation constant. Naturally, for truncation order $N$ there are $M=2N+1$ modes.

For TM polarization, the physical equation \eqref{eq:TM_basic_short} leads to the Fourier–modal matrix
\begin{equation}\label{eq:TM_modal}
  \llbracket\epsilon\rrbracket
  \big(I - K\,\llbracket\epsilon^{-1}\rrbracket\,K\big)\,[H_y]
  = \eta^{\,2}\,[H_y].
\end{equation}
We refer to this formulation as the \emph{first truncation scheme}.

Classical Toeplitz asymptotics yield the approximation $\llbracket\epsilon^{-1}\rrbracket \approx \llbracket\epsilon\rrbracket^{-1}$ for large truncation orders~\cite{Gray2006Book}, producing the simpler eigenvalue problem---the \emph{second scheme}---which often yields improved numerical behavior~\cite{Moharam1995a, Garnet1996}:
\begin{equation}\label{eq:TM_modal_classical}
  \llbracket\epsilon\rrbracket
  \big(I - K\,\llbracket\epsilon\rrbracket^{-1} K\big)\,[H_y]
  = \eta^{\,2}\,[H_y].
\end{equation}
The factorization rules of~\cite{Li1996a} provide a more accurate representation of the TM operator, leading to the \emph{third scheme}
\begin{equation}\label{eq:TM_modal_factorized}
  \llbracket\epsilon^{-1}\rrbracket^{-1}
  \big(I - K\,\llbracket\epsilon\rrbracket^{-1} K\big)\,[H_y]
  = \eta^{\,2}\,[H_y].
\end{equation}
This will be our standard TM formulation unless otherwise noted.

\medskip
\subsection{Example: Lamellar Grating}\label{subsec:example}
We consider the TM benchmark configuration of \cite{Popov2004a}, illustrated in Fig.~\ref{fig:setting}. The unit cell consists of an air region with $\epsilon_1 = 1$ and a lossless metallic region with $\epsilon_2 = -100$, with groove width~$g$; in our examples we take $g = \ell/2$. The Bloch phase is specified by $k_{x,0}/k_0 = \sin(\theta)$, and we set $\theta = 30^\circ$. We choose $\lambda_{\text{vac}} = 632.8\,\mathrm{nm}$ and $\ell = 500\,\mathrm{nm}$, and compute the eigenvalues $\eta^{\,2}$. These eigenvalues depend parametrically on $(\epsilon_1,\epsilon_2)$, the ratios $g/\ell$ and $\ell/\lambda_{\text{vac}}$, and the incidence angle~$\theta$. For comparison, we also consider a hypothetical dielectric grating of the same geometry, with $\epsilon_1 = 1$ and $\epsilon_2 = 100$, so that the contrast is of comparable magnitude.

Although the full diffraction problem proceeds by matching modal expansions across layers using the stable formulations of \cite{Moharam1995b, Li1996b, Li2003a}, our purpose here is only to examine the eigenvalue structure associated with \eqref{eq:TM_modal_classical} and \eqref{eq:TM_modal_factorized}. Moreover, since the phenomenon of interest manifests most severely in lossless configurations, we restrict attention to real-valued $\epsilon_2$; adding a small imaginary part does not change the qualitative behavior relevant to our discussion.

\begin{figure}[tbp]
    \centering
    \begin{subfigure}{0.55\textwidth}
        \centering
        \includegraphics[width=\textwidth]{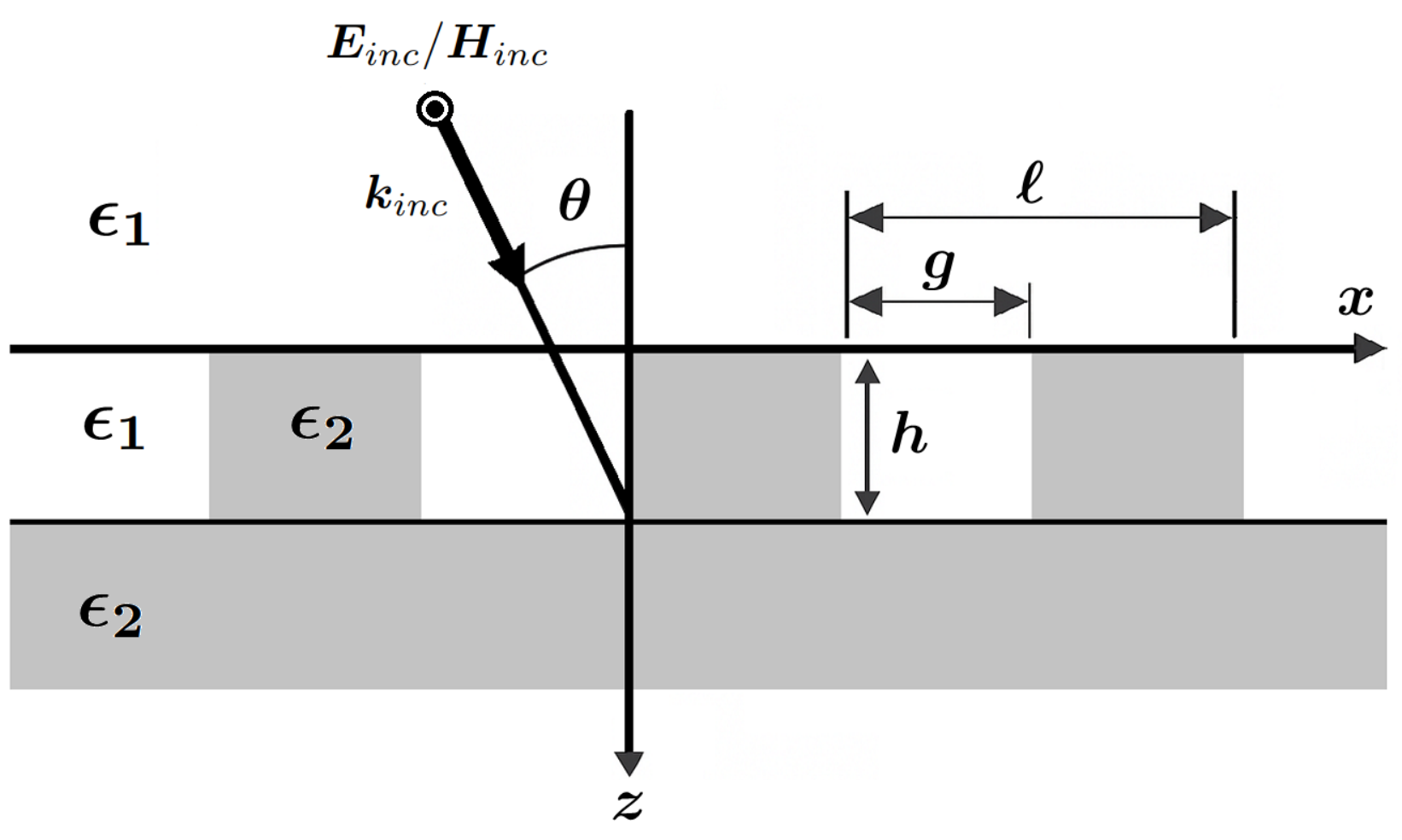}
        \caption{}\label{fig:setting}
    \end{subfigure}
    \hfill
    \begin{subfigure}{0.42\textwidth}
        \centering
        \includegraphics[width=\textwidth]{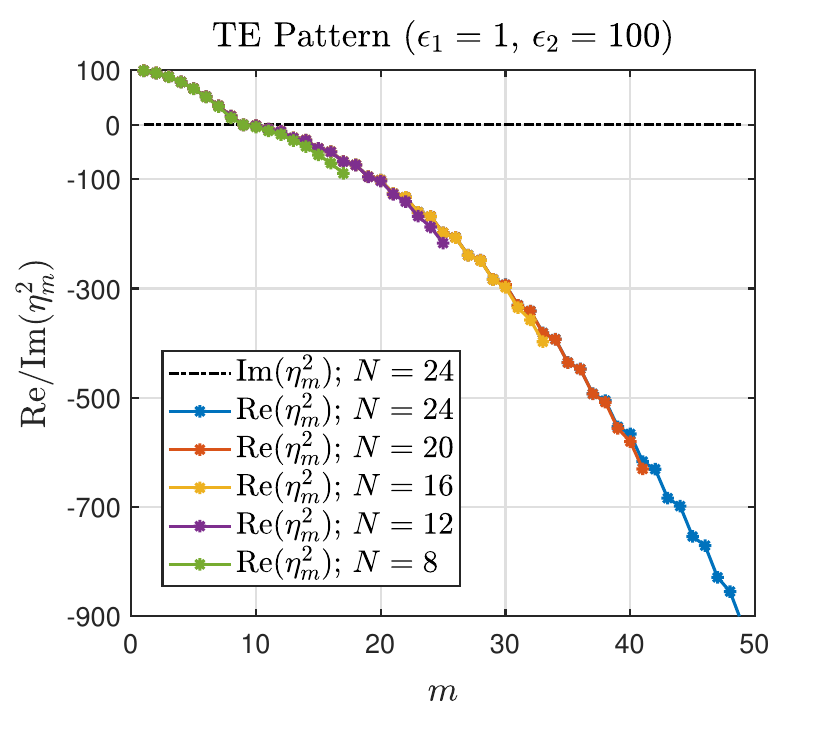}
        \caption{}\label{fig:patt-N}
    \end{subfigure}

    \caption{ 
    (a) Geometry of a lamellar grating problem under TE/TM incidence.  
    (b) Development of the TE eigenvalue pattern in a dielectric grating for $N=8\sim24$.}
\end{figure}

Following \cite{Faghihifar2019, Faghihifar2019conf}, it is useful for assessment or comparison to sort the eigenvalues according to $\operatorname{Re}(\eta^{2})$ in descending order and then, in this ordering, to plot their real and corresponding imaginary parts. For geometrically simple and regular configurations, the plot of the real parts exhibits a stable pattern with no finite accumulation points, which settles from the left as $N \to \infty$, while the imaginary parts either vanish or remain negligible. Fig.~\ref{fig:patt-N} illustrates the formation of this pattern as $N$ increases, for the lamellar dielectric grating described earlier under TE incidence.

Eigenvalues with $\operatorname{Re}(\eta^{2}) > 0$ correspond to propagating modes and dominate both the physical behaviour and the numerical approximation, as they carry energy. Those with $\operatorname{Re}(\eta^{2}) < 0$ are evanescent modes. These eigenvalues extend unboundedly toward $-\infty$ and are therefore infinite in number. The modes occupying the far-right tail of each truncated pattern are most sensitive to truncation errors; however, they play little role in the resulting diffraction solution.

For discrete dielectric gratings, the eigenvalue pattern consists of several interconnected curve-like branches, each originating near one of the grating permittivity values. Consequently, in the dielectric case the eigenvalues are bounded above by $\epsilon_{\max}$, ensuring that only finitely many propagating modes exist. However, metallic gratings under TM incidence can support genuine modes that violate this upper bound \cite{Tishchenko2005, Foresti2006}. Moreover, spurious modes appear precisely in this physically significant region and therefore cannot be distinguished \emph{a priori}. To facilitate comparison across truncation orders, we index these outlying modes as $m = 0, -1, -2, \dots$ so that the ordering of the pattern remains consistent as $N$ varies.

\begin{figure}[tbp]
    \centering
    \begin{subfigure}{0.49\textwidth}
        \includegraphics[width=\textwidth]{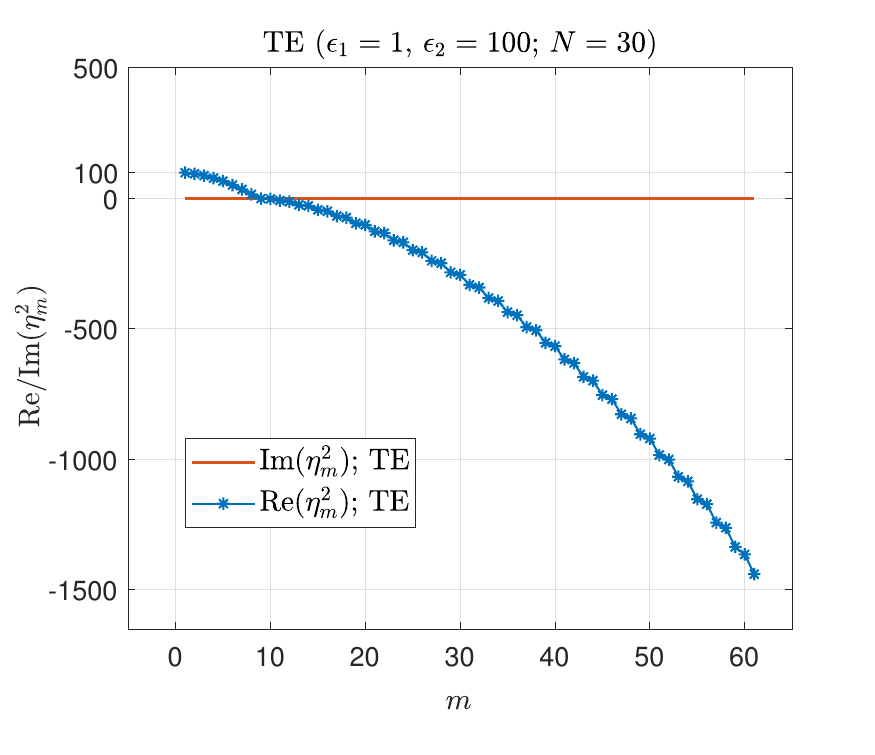}
        \caption{}
        \label{fig:4g-te-d}
    \end{subfigure}
    \hfill
    \begin{subfigure}{0.49\textwidth}
        \includegraphics[width=\textwidth]{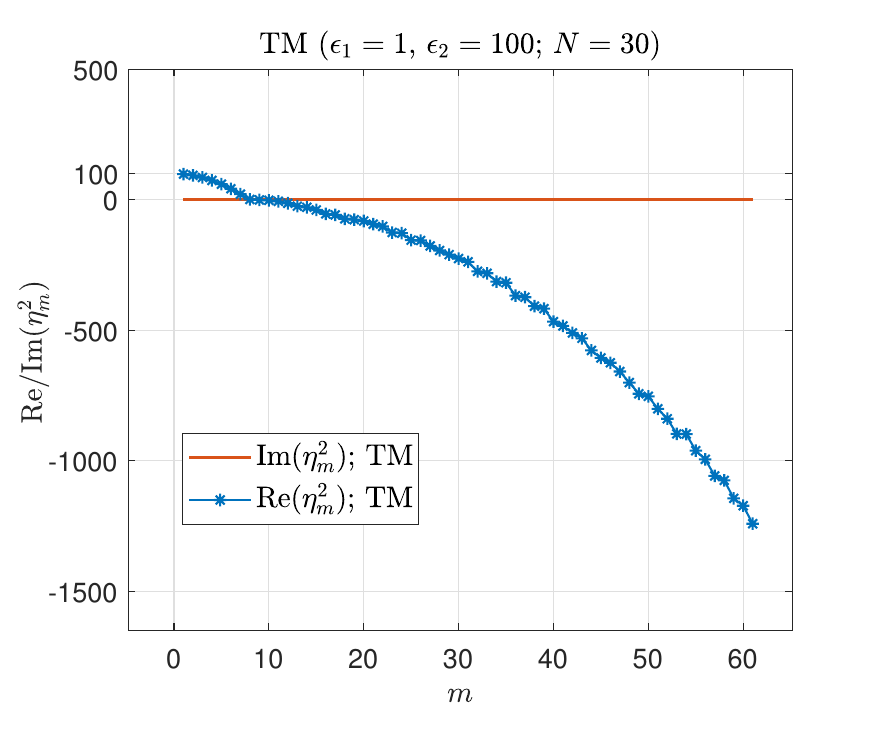}
        \caption{}
        \label{fig:4g-tm-d}
    \end{subfigure}
    \begin{subfigure}{0.49\textwidth}
        \includegraphics[width=\textwidth]{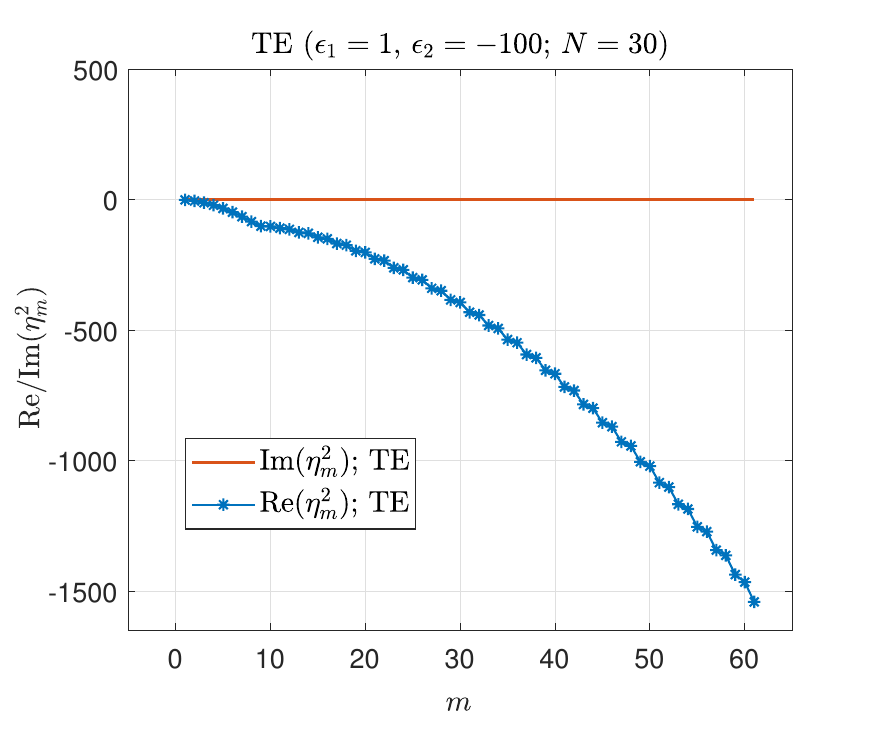}
        \caption{}
        \label{fig:4g-te-m}
    \end{subfigure}
    \hfill
    \begin{subfigure}{0.49\textwidth}
        \includegraphics[width=\textwidth]{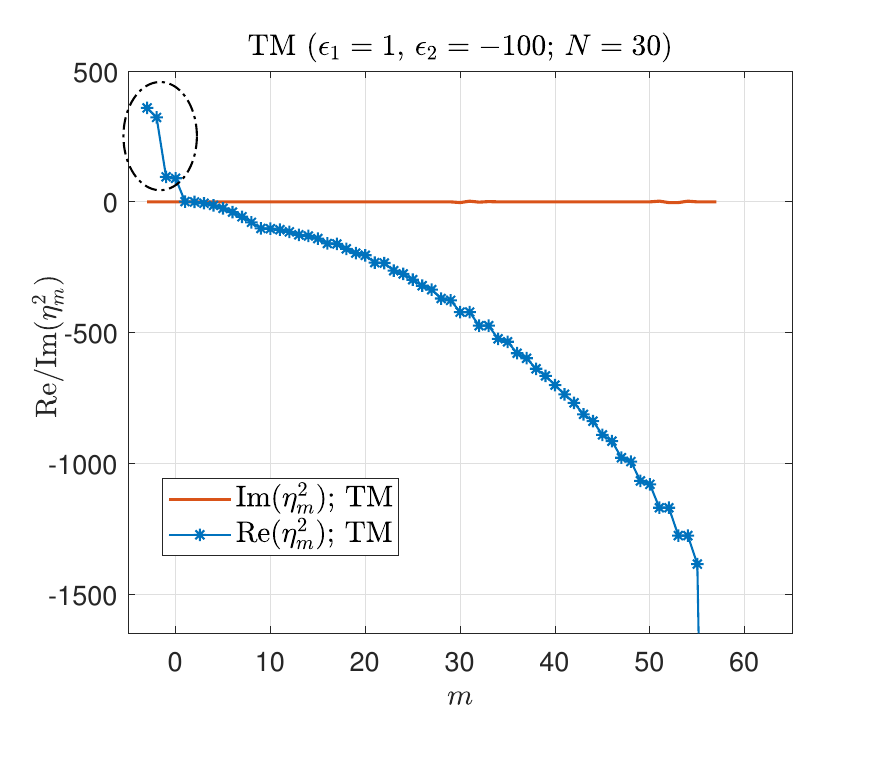}
        \caption{}
        \label{fig:4g-tm-m}
    \end{subfigure}

    \caption{Eigenvalue graphs for various lamellar grating settings:
    (a) TE/dielectric, (b) TM/dielectric, 
    (c) TE/metallic, and (d) TM/metallic.}
    \label{fig:4_graphs}
\end{figure}

Fig.~\ref{fig:4_graphs} displays the eigenvalue graphs for the four cases at truncation order $N=30$. Among these, only the metallic TM problem deviates from the expected pattern, as seen in Fig.~\ref{fig:4g-tm-m}, where spurious eigenvalues appear. This confirms that the emergence of spurious modes does not arise from the mere presence of a high dielectric contrast. Another distinctive feature occurs in the imaginary parts: for the first three cases they vanish identically, whereas nonzero values may occur in the metallic TM case \cite{Foresti2006}. However, this effect is not clearly visible in the plots and is not the focus of our study.

Unlike the permittivity contrast, the emergence of outlying modes depends sensitively on the chosen truncation scheme. Fig.~\ref{fig:scm} examines how the three schemes introduced earlier affect the behaviour of spurious modes. In Fig.~\ref{fig:scm-m}, we plot the magnitude $|\eta_m|^2$---essentially $\operatorname{Re}(\eta_m^{\,2})$---on a logarithmic scale for $N=30$ under each scheme. Physical modes should agree across all schemes as $N$ increases; for example, the mode indexed by $m=1$ appears consistently as the dominant physical mode.

The first truncation scheme produces the largest number of---presumably---spurious modes, exhibiting extremely large eigenvalues. The second scheme produces only one outlier, but with an exceptionally large value. The third scheme, the modern standard, still produces four spurious eigenvalues, smaller than those of the other schemes but still very large compared with the physical $m=1$ eigenvalue. To trace their progression, Fig.~\ref{fig:scm-N} shows the evolution of these modes for $N = 16 \sim 40$, with each scheme represented by a distinct colour. The qualitative conclusions drawn at $N=30$ persist throughout this range.

\begin{figure}[tbp]
    \centering
    \begin{subfigure}{0.47\textwidth}
        \centering
        \includegraphics[width=\textwidth]{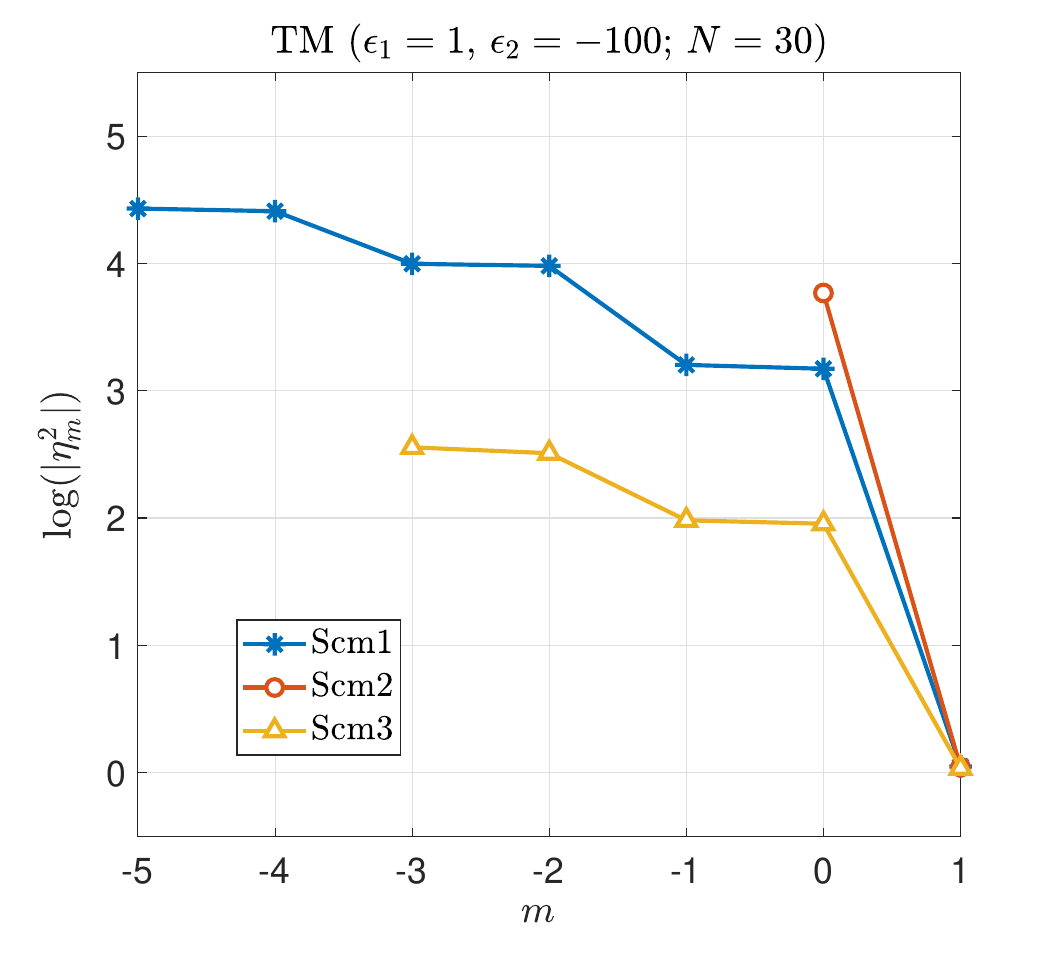}
        \caption{}\label{fig:scm-m}
    \end{subfigure}
    \hfill
    \begin{subfigure}{0.47\textwidth}
        \centering
        \includegraphics[width=\textwidth]{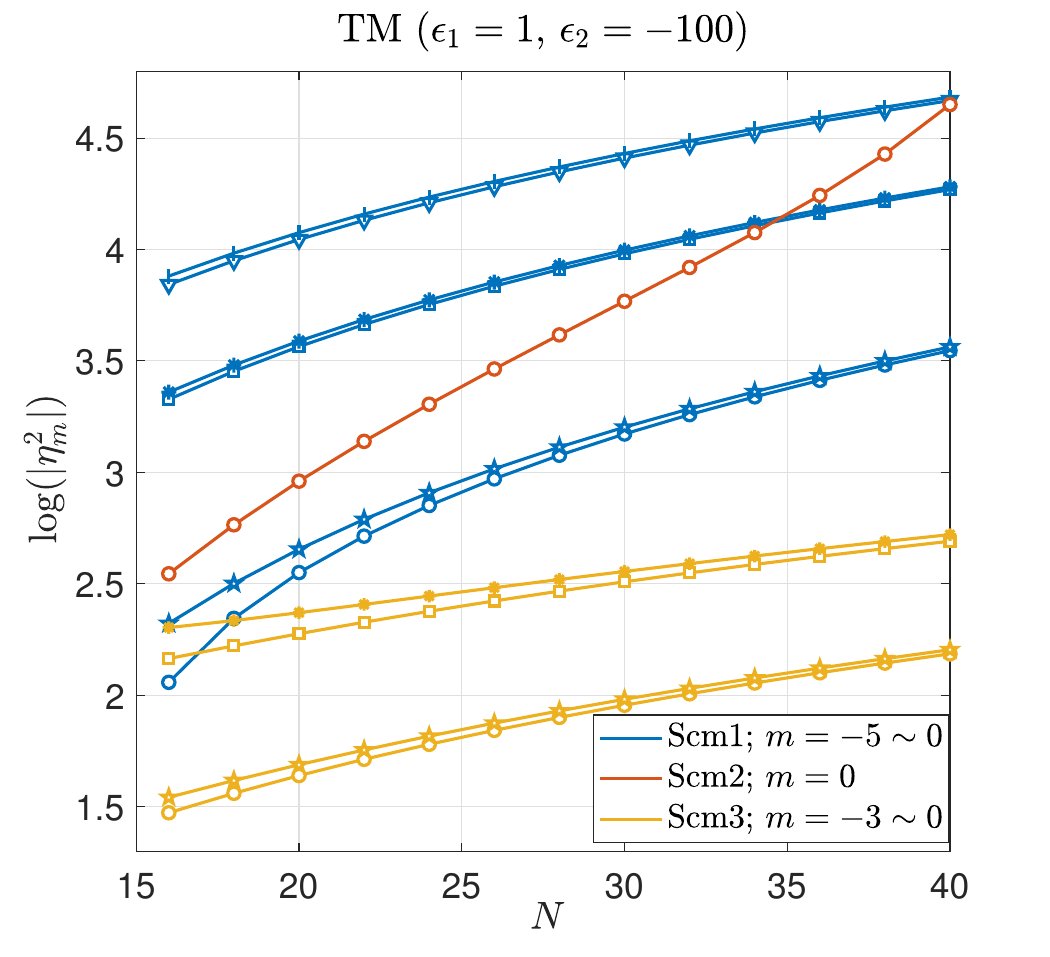}
        \caption{}\label{fig:scm-N}
    \end{subfigure}

    \caption{Comparison of spurious eigenvalue magnitudes across truncation schemes.  
    (a) Spurious eigenvalues at $N=30$.  
    (b) Evolution of their magnitudes for $N=16\sim40$.}
    \label{fig:scm}
\end{figure}

\begin{figure}[tbp]
    \centering
    \begin{subfigure}{0.49\textwidth}
        \centering
        \includegraphics[width=\textwidth]{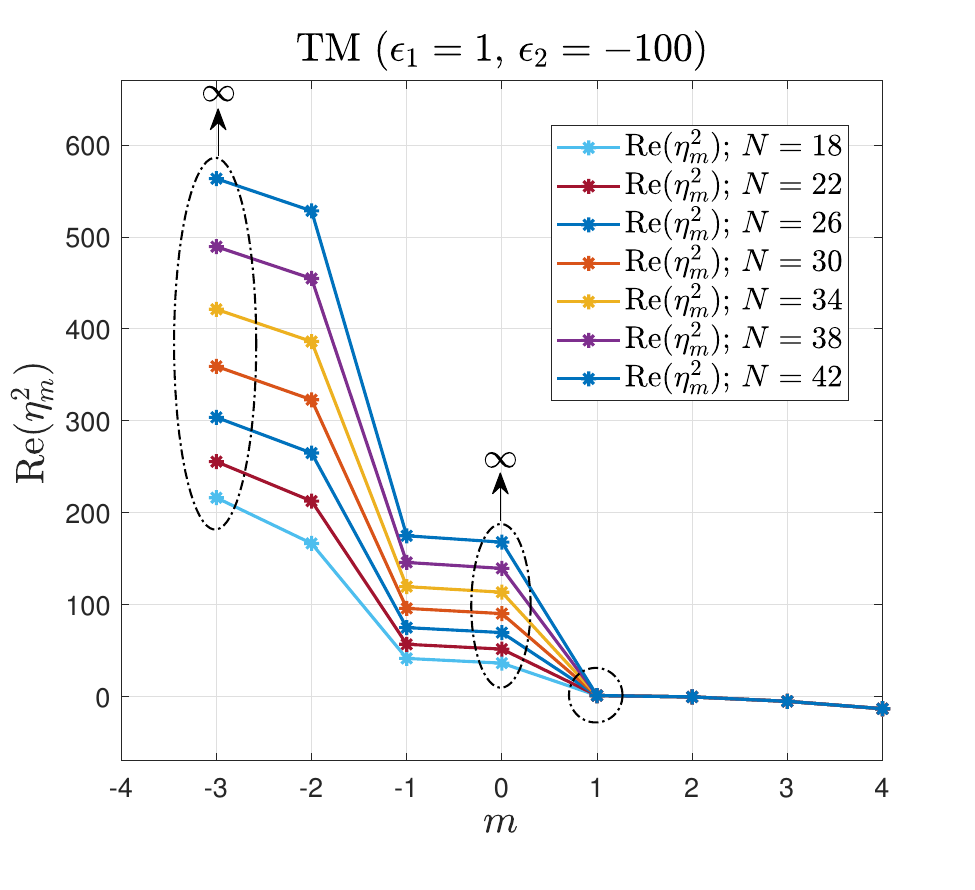}
        \caption{}\label{fig:sp-m}
    \end{subfigure}
    \hfill
    \begin{subfigure}{0.49\textwidth}
        \centering
        \includegraphics[width=\textwidth]{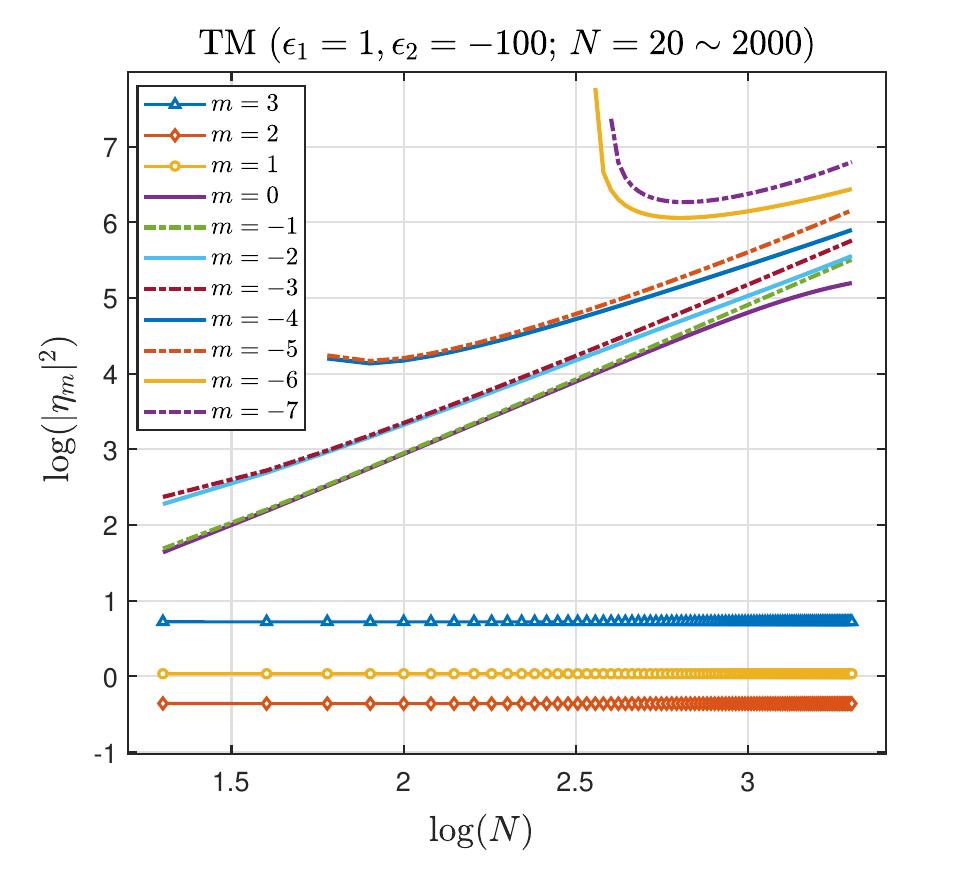}
        \caption{}\label{fig:sp-N}
    \end{subfigure}

    \caption{Growth of spurious modes with truncation order:  
    (a) zoom on the spurious region for $N=18\sim 42$;  
    (b) log--log scale evolution of selected physical and spurious modes for $N=20\sim2000$.}
    \label{fig:sp}
\end{figure}

Another observation from Fig.~\ref{fig:scm-m}, and the most important one of this section, concerns the \emph{growth} of spurious modes with the truncation order. Focusing on the standard truncation scheme, Fig.~\ref{fig:sp-m} zooms in on the spurious region of the eigenvalue pattern. While the physical part of the pattern converges as $N$ increases from $18$ to $42$, the spurious modes appear to grow without bound.

To examine this over a wider range, Fig.~\ref{fig:sp-N} plots the magnitudes of the non-physical eigenvalues, together with three physical modes ($m=1,2,3$), for $N = 20 \sim 2000$ on a log--log scale. As $N$ crosses certain thresholds, new spurious eigenvalues emerge with increasingly large magnitudes, suggesting that the number of spurious modes is plausibly infinite. Moreover, although the growth of individual spurious eigenvalues may not persist uniformly for all $N$, the clearly unbounded divergence observed in practice is consistent with the behaviour expected when the coefficients change sign. The log--log scaling is consistent with polynomial growth in $N$, though we do not attempt to determine a rate.

This empirical behavior, together with our analytical left half-plane boundedness result, yields a simple and robust rule of thumb for identifying spurious modes in the metallic TM lamellar-grating problem. Since $\eta^{2} = -\lambda/k_0^{2}$ in the SL formulation, boundedness in the left half-plane implies the existence of $R_0>0$ such that every eigenvalue with $\operatorname{Re}(\eta^{2})>0$ satisfies $|\eta^{2}|<R_0/k_0^{2}$. Consequently, any computed mode whose $\operatorname{Re}(\eta^{2})$ grows without bound as $N$ increases cannot converge to an eigenvalue of \eqref{eq:SL_form}, and is therefore non-physical. This is the criterion announced in Section~\ref{sec:intro}: divergence in $N$ is sufficient to certify a mode as spurious, independently of the truncation scheme or the observed growth rate.

\section{Conclusion}
\label{sec:conclusion}
We establish a left half-plane boundedness result for the eigenvalues of a class of SL problems in which the principal coefficient and the weight are coupled through a positive ratio. TM diffraction in lamellar gratings provides a physically important instance of this framework, corresponding to the case $w=p$. The result implies that only finitely many eigenvalues may lie in the open left half-plane, and the exclusion region admits an explicit bound in terms of the coefficients and the interface data. This yields a criterion for identifying spurious modes in numerical eigenvalue solvers, particularly in Fourier modal methods, and the accompanying numerical examples illustrate its practical applicability.

\section*{Acknowledgment}
ChatGPT (OpenAI) and Claude (Anthropic) were used for language polishing, stylistic refinement, and assistance with the clarity and presentation of the mathematical exposition by identifying redundancies and clarifying definitions and assumptions. The mathematical results, proofs, numerical implementation, and conclusions are the author's own.

\vskip6pc

\bibliographystyle{siamplain}

\begin{thebibliography}{10}

\bibitem{Abasheeva1997}
{\sc N.~Abasheeva and S.~Pyatkov}, {\em Counterexamples in indefinite
  sturm-liouville problems}, Siberian Advances in Mathematics, 7 (1997),
  pp.~1--8.

\bibitem{Popov2014Book}
{\sc T.~Antonakakis, F.~I. Baida, A.~Belkhir, K.~Cherednichenko, S.~Cooper,
  R.~Craster, G.~Dem{\'e}sy, J.~Desanto, G.~Granet, B.~Gralak, et~al.}, {\em
  Gratings: Theory and numeric applications}, AMU,(PUP), CNRS, ECM, 2014.

\bibitem{Beals1985}
{\sc R.~Beals}, {\em Indefinite sturm-liouville problems and half-range
  completeness}, Journal of Differential Equations, 56 (1985), pp.~391--407.

\bibitem{Behrndt2013a}
{\sc J.~Behrndt, S.~Chen, F.~Philipp, and J.~Qi}, {\em Bounds on non-real
  eigenvalues of indefinite sturm-liouville problems}, PAMM, 13 (2013),
  pp.~525--526.

\bibitem{Behrndt2013b}
{\sc J.~Behrndt, S.~Chen, F.~Philipp, and J.~Qi}, {\em Estimates on the
  non-real eigenvalues of regular indefinite sturm--liouville problems},
  Proceedings of the Royal Society of Edinburgh Section A: Mathematics, 144
  (2014), pp.~1113--1126.

\bibitem{Behrndt2013c}
{\sc J.~Behrndt, F.~Philipp, and C.~Trunk}, {\em Bounds on the non-real
  spectrum of differential operators with indefinite weights}, Mathematische
  Annalen, 357 (2013), pp.~185--213.

\bibitem{Behrndt2018}
{\sc J.~Behrndt, P.~Schmitz, and C.~Trunk}, {\em Spectral bounds for singular
  indefinite sturm-liouville operators with $l^1$--potentials}, Proceedings of
  the American Mathematical Society, 146 (2018), pp.~3935--3942.

\bibitem{Binding2004}
{\sc P.~Binding and B.~{\'C}urgus}, {\em A counterexample in sturm--liouville
  completeness theory}, Proceedings of the Royal Society of Edinburgh Section
  A: Mathematics, 134 (2004), pp.~241--248.

\bibitem{Birkhoff1989Book}
{\sc G.~Birkhoff and G.~Rota}, {\em Ordinary Differential Equations},
  Introductions to higher mathematics, Wiley, 1989.

\bibitem{Bonod2015}
{\sc N.~Bonod and J.~Neauport}, {\em Diffraction gratings: from principles to
  applications in high-intensity lasers}, 2016.

\bibitem{Botten1981a}
{\sc I.~Botten, M.~Craig, R.~McPhedran, J.~Adams, and J.~Andrewartha}, {\em The
  dielectric lamellar diffraction grating}, Optica Acta: International Journal
  of Optics, 28 (1981), pp.~413--428.

\bibitem{Botten1981c}
{\sc L.~Botten, M.~Craig, and R.~McPhedran}, {\em Highly conducting lamellar
  diffraction gratings}, Optica Acta: International Journal of Optics, 28
  (1981), pp.~1103--1106.

\bibitem{Botten1981b}
{\sc L.~Botten, M.~Craig, R.~McPhedran, J.~Adams, and J.~Andrewartha}, {\em The
  finitely conducting lamellar diffraction grating}, Optica Acta: International
  Journal of Optics, 28 (1981), pp.~1087--1102.

\bibitem{Botten1985}
{\sc L.~Botten and R.~McPhedran}, {\em Completeness and modal expansion methods
  in diffraction theory}, Optica Acta: International Journal of Optics, 32
  (1985), pp.~1479--1488.

\bibitem{Burckhardt1966}
{\sc C.~Burckhardt}, {\em Diffraction of a plane wave at a sinusoidally
  stratified dielectric grating}, Journal of the Optical Society of America, 56
  (1966), pp.~1502--1508.

\bibitem{Chicone2006Book}
{\sc C.~Chicone}, {\em Ordinary differential equations with applications},
  Springer, 2006.

\bibitem{Chiou2009}
{\sc Y.-P. Chiou, W.-L. Yeh, and N.-Y. Shih}, {\em Analysis of highly
  conducting lamellar gratings with multidomain pseudospectral method}, 2009.

\bibitem{Curgus2013}
{\sc B.~{\'C}urgus, A.~Fleige, and A.~Kostenko}, {\em The riesz basis property
  of an indefinite sturm--liouville problem with non-separated boundary
  conditions}, Integral Equations and Operator Theory, 77 (2013), pp.~533--557.

\bibitem{Curgus1989}
{\sc B.~Curgus and H.~Langer}, {\em A krein space approach to symmetric
  ordinary differential operators with an indefinite weight function}, J.
  Differential Equations, 79 (1989), pp.~31--61.

\bibitem{Dudley2015Book}
{\sc D.~Dudley, I.~Antennas, and P.~Society}, {\em Mathematical Foundations for
  Electromagnetic Theory}, IEEE Electromagnetic waves series, IEEE Press, 1994.

\bibitem{Edee2011}
{\sc K.~Edee, I.~Fenniche, G.~Granet, and B.~Guizal}, {\em Modal method based
  on subsectional gegen-bauer polynomial expansion for lamellar gratings:
  Weighting function, convergence and stability}, Progress In Electromagnetics
  Research, 133 (2013), pp.~17--35.

\bibitem{Edee2015}
{\sc K.~Edee and J.~Plumey}, {\em Numerical scheme for the modal method based
  on subsectional gegenbauer polynomial expansion: application to biperiodic
  binary grating}, Journal of the Optical Society of America A, 32 (2015),
  pp.~402--410.

\bibitem{Everitt1982}
{\sc W.~N. Everitt}, {\em On the transformation theory of ordinary second-order
  linear symmetric differential expressions}, Czechoslovak Mathematical
  Journal, 32 (1982), pp.~275--306.

\bibitem{Faghihifar2022}
{\sc E.~Faghihifar and M.~Akbari}, {\em Exclusive robustness of gegenbauer
  method to truncated convolution errors}, Journal of Computational Physics,
  452 (2022), p.~110911.

\bibitem{Faghihifar2019conf}
{\sc E.~Faghihifar, M.~Akbari, and S.~A.~H. Nekuee}, {\em Fast estimation of
  propagation constants in lamellar gratings needless of solving the eigenvalue
  equation}, in 2019 27th Iranian Conference on Electrical Engineering (ICEE),
  IEEE, 2019, pp.~1342--1346.

\bibitem{Faghihifar2019}
{\sc E.~Faghihifar, M.~Akbari, and S.~A.~H. Nekuee}, {\em Fast estimation of
  propagation constants in crossed gratings}, Journal of Optics, 22 (2020),
  p.~025001.

\bibitem{Fleige1998}
{\sc A.~Fleige}, {\em A counterexample to completeness properties for
  indefinite st urm—liouville problems}, Mathematische Nachrichten, 190
  (1998), pp.~123--128.

\bibitem{Foresti2006}
{\sc M.~Foresti, L.~Menez, and A.~V. Tishchenko}, {\em Modal method in deep
  metal--dielectric gratings: the decisive role of hidden modes}, Journal of
  the Optical Society of America A, 23 (2006), pp.~2501--2509.

\bibitem{Garnet1999}
{\sc G.~Granet}, {\em Reformulation of the lamellar grating problem through the
  concept of adaptive spatial resolution}, Journal of the Optical Society of
  America A, 16 (1999), pp.~2510--2516.

\bibitem{Garnet2014}
{\sc G.~Granet}, {\em Efficient implementation of b-spline modal method for
  lamellar gratings}, 2014.

\bibitem{Garnet2010}
{\sc G.~Granet, L.~B. Andriamanampisoa, K.~Raniriharinosy, A.~M. Armeanu, and
  K.~Edee}, {\em Modal analysis of lamellar gratings using the moment method
  with subsectional basis and adaptive spatial resolution}, Journal of the
  Optical Society of America A, 27 (2010), pp.~1303--1310.

\bibitem{Garnet1996}
{\sc G.~Granet and B.~Guizal}, {\em Efficient implementation of the
  coupled-wave method for metallic lamellar gratings in tm polarization},
  Journal of the Optical Society of America A, 13 (1996), pp.~1019--1023.

\bibitem{Gray2006Book}
{\sc R.~M. Gray et~al.}, {\em Toeplitz and circulant matrices: A review},
  Foundations and Trends{\textregistered} in Communications and Information
  Theory, 2 (2006), pp.~155--239.

\bibitem{Guizal2009}
{\sc B.~Guizal, H.~Yala, and D.~Felbacq}, {\em Reformulation of the eigenvalue
  problem in the fourier modal method with spatial adaptive resolution}, Optics
  letters, 34 (2009), pp.~2790--2792.

\bibitem{Gundu2010b}
{\sc K.~M. Gundu and A.~Mafi}, {\em Constrained least squares fourier modal
  method for computing scattering from metallic binary gratings}, Journal of
  the Optical Society of America A, 27 (2010), pp.~2375--2380.

\bibitem{Gundu2010a}
{\sc K.~M. Gundu and A.~Mafi}, {\em Reliable computation of scattering from
  metallic binary gratings using fourier-based modal methods}, Journal of the
  Optical Society of America A, 27 (2010), pp.~1694--1700.

\bibitem{Hanson2013Book}
{\sc G.~Hanson and A.~Yakovlev}, {\em Operator Theory for Electromagnetics: An
  Introduction}, Springer New York, 2013.

\bibitem{Haupt1914}
{\sc O.~Haupt}, {\em {\"U}ber eine methode zum beweise von
  oszillationstheoremen}, Mathematische Annalen, 76 (1914), pp.~67--104.

\bibitem{Hilb1911}
{\sc E.~Hilb}, {\em {\"U}ber reihenentwicklungen nach den eigenfunktionen
  linearer differentialgleichungen 2ter ordnung}, Mathematische Annalen, 71
  (1911), pp.~76--87.

\bibitem{Joannopoulos2008Book}
{\sc J.~D. Joannopoulos, P.~R. Villeneuve, and S.~Fan}, {\em Photonic
  crystals}, vol.~102, Elsevier, 1997.

\bibitem{Kaper1984}
{\sc H.~G. Kaper, M.~K. Kwong, C.~Lekkerkerker, and A.~Zettl}, {\em Full-and
  partial-range eigenfunction expansions for sturm-liouville problems with
  indefinite weights}, Proceedings of the Royal Society of Edinburgh Section A:
  Mathematics, 98 (1984), pp.~69--88.

\bibitem{Kasper1973}
{\sc F.~G. Kaspar}, {\em Diffraction by thick, periodically stratified gratings
  with complex dielectric constant}, J. Opt. Soc. Am., 63 (1973), pp.~37--45.

\bibitem{Khavasi2009}
{\sc A.~Khavasi and K.~Mehrany}, {\em Artifact-free analysis of highly
  conducting binary gratings by using the legendre polynomial expansion
  method}, Journal of the Optical Society of America A, 26 (2009),
  pp.~1467--1471.

\bibitem{Khavasi2008}
{\sc A.~Khavasi, K.~Mehrany, and A.~M. Jazayeri}, {\em Study of the numerical
  artifacts in differential analysis of highly conducting gratings}, Optics
  letters, 33 (2008), pp.~159--161.

\bibitem{Khavasi2007}
{\sc A.~Khavasi, K.~Mehrany, and B.~Rashidian}, {\em Three-dimensional
  diffraction analysis of gratings based on legendre expansion of
  electromagnetic fields}, Journal of the Optical Society of America B, 24
  (2007), pp.~2676--2685.

\bibitem{Kikonko2016}
{\sc M.~Kikonko and A.~B. Mingarelli}, {\em Bounds on real and imaginary parts
  of non-real eigenvalues of a non-definite sturm--liouville problem}, Journal
  of Differential Equations, 261 (2016), pp.~6221--6232.

\bibitem{Kim2012}
{\sc H.~Kim, G.-W. Park, and C.-S. Kim}, {\em Investigation of the convergence
  behavior with fluctuation features in the fourier modal analysis of a
  metallic grating}, Journal of the Optical Society of Korea, 16 (2012),
  pp.~196--202.

\bibitem{Knop1978}
{\sc K.~Knop}, {\em Rigorous diffraction theory for transmission phase gratings
  with deep rectangular grooves}, Journal of the Optical Society of America, 68
  (1978), pp.~1206--1210.

\bibitem{Kocabas2009}
{\sc {\c{S}}.~E. Kocaba{\c{s}}, G.~Veronis, D.~A. Miller, and S.~Fan}, {\em
  Modal analysis and coupling in metal-insulator-metal waveguides}, Physical
  Review B—Condensed Matter and Materials Physics, 79 (2009), p.~035120.

\bibitem{Kogelnik1969}
{\sc H.~Kogelnik}, {\em Coupled wave theory for thick hologram gratings}, Bell
  System Technical Journal, 48 (1969), pp.~2909--2947.

\bibitem{Kong2003}
{\sc Q.~Kong, H.~Wu, A.~Zettl, and M.~M{\"o}ller}, {\em Indefinite
  sturm--liouville problems}, Proceedings of the Royal Society of Edinburgh
  Section A: Mathematics, 133 (2003), pp.~639--652.

\bibitem{Lalanne2000}
{\sc P.~Lalanne and J.-P. Hugonin}, {\em Numerical performance of
  finite-difference modal methods for the electromagnetic analysis of
  one-dimensional lamellar gratings}, Journal of the Optical Society of America
  A, 17 (2000), pp.~1033--1042.

\bibitem{Lalanne1996}
{\sc P.~Lalanne and G.~M. Morris}, {\em Highly improved convergence of the
  coupled-wave method for tm polarization}, Journal of the Optical Society of
  America A, 13 (1996), pp.~779--784.

\bibitem{Li1993b}
{\sc L.~Li}, {\em A modal analysis of lamellar diffraction gratings in conical
  mountings}, Journal of Modern Optics, 40 (1993), pp.~553--573.

\bibitem{Li1996b}
{\sc L.~Li}, {\em Formulation and comparison of two recursive matrix algorithms
  for modeling layered diffraction gratings}, Journal of the Optical Society of
  America A, 13 (1996), pp.~1024--1035.

\bibitem{Li1996a}
{\sc L.~Li}, {\em Use of fourier series in the analysis of discontinuous
  periodic structures}, Journal of the Optical Society of America A, 13 (1996),
  pp.~1870--1876.

\bibitem{Li1997}
{\sc L.~Li}, {\em New formulation of the fourier modal method for crossed
  surface-relief gratings}, J. Opt. Soc. Am. A, 14 (1997), pp.~2758--2767.

\bibitem{Li1999}
{\sc L.~Li}, {\em Justification of matrix truncation in the modal methods of
  diffraction gratings}, Journal of Optics A: Pure and Applied Optics, 1
  (1999), p.~531.

\bibitem{Li2001Chap}
{\sc L.~Li}, {\em Mathematical reflections on the fourier modal method in
  grating theory}, in Mathematical modeling in optical science, SIAM, 2001,
  pp.~111--139.

\bibitem{Li2003a}
{\sc L.~Li}, {\em Note on the s-matrix propagation algorithm}, Journal of the
  Optical Society of America A, 20 (2003), pp.~655--660.

\bibitem{Li2012}
{\sc L.~Li}, {\em Field singularities at lossless metal-dielectric
  arbitrary-angle edges and their ramifications to the numerical modeling of
  gratings.}, Journal of the Optical Society of America. A, Optics, image
  science, and vision, 29 (2012), pp.~593--604.

\bibitem{Li2011}
{\sc L.~Li and G.~Granet}, {\em Field singularities at lossless
  metal-dielectric right-angle edges and their ramifications to the numerical
  modeling of gratings}, Journal of the Optical Society of America A, 28
  (2011), pp.~738--746.

\bibitem{Li1993a}
{\sc L.~Li and C.~W. Haggans}, {\em Convergence of the coupled-wave method for
  metallic lamellar diffraction gratings}, Journal of the Optical Society of
  America A, 10 (1993), pp.~1184--1189.

\bibitem{Lyndin2007}
{\sc N.~M. Lyndin, O.~Parriaux, and A.~V. Tishchenko}, {\em Modal analysis and
  suppression of the fourier modal method instabilities in highly conductive
  gratings}, Journal of the Optical Society of America A, 24 (2007),
  pp.~3781--3788.

\bibitem{Magnusson1978}
{\sc R.~Magnusson and T.~Gaylord}, {\em Equivalence of multiwave coupled-wave
  theory and modal theory for periodic-media diffraction}, Journal of the
  Optical Society of America, 68 (1978), pp.~1777--1779.

\bibitem{Mei2014}
{\sc Y.~Mei, H.~Liu, and Y.~Zhong}, {\em Treatment of nonconvergence of fourier
  modal method arising from irregular field singularities at lossless
  metal-dielectric right-angle edges}, Journal of the Optical Society of
  America A, 31 (2014), pp.~900--906.

\bibitem{Mingarelli1982}
{\sc A.~B. Mingarelli}, {\em Indefinite sturm-liouville problems}, in Ordinary
  and Partial Differential Equations: Proceedings of the Seventh Conference
  Held at Dundee, Scotland, March 29--April 2, 1982, Springer, 2006,
  pp.~519--528.

\bibitem{Mingarelli2011}
{\sc A.~B. Mingarelli}, {\em A survey of the regular weighted sturm-liouville
  problem-the non-definite case}, arXiv preprint arXiv:1106.6013,  (2011).

\bibitem{Moharam1982}
{\sc M.~Moharam and T.~K. Gaylord}, {\em Diffraction analysis of dielectric
  surface-relief gratings}, Journal of the Optical Society of America, 72
  (1982), pp.~1385--1392.

\bibitem{Moharam1983}
{\sc M.~Moharam and T.~K. Gaylord}, {\em Three-dimensional vector coupled-wave
  analysis of planar-grating diffraction}, Journal of the Optical Society of
  America, 73 (1983), pp.~1105--1112.

\bibitem{Moharam1986}
{\sc M.~Moharam and T.~K. Gaylord}, {\em Rigorous coupled-wave analysis of
  metallic surface-relief gratings}, Journal of the Optical Society of America
  A, 3 (1986), pp.~1780--1787.

\bibitem{Moharam1995b}
{\sc M.~Moharam, D.~A. Pommet, E.~B. Grann, and T.~K. Gaylord}, {\em Stable
  implementation of the rigorous coupled-wave analysis for surface-relief
  gratings: enhanced transmittance matrix approach}, Journal of the Optical
  Society of America A, 12 (1995), pp.~1077--1086.

\bibitem{Moharam1995a}
{\sc M.~G. Moharam, E.~B. Grann, D.~A. Pommet, and T.~K. Gaylord}, {\em
  Formulation for stable and efficient implementation of the rigorous
  coupled-wave analysis of binary gratings}, Journal of the Optical Society of
  America A, 12 (1995), pp.~1068--1076.

\bibitem{Morf1995}
{\sc R.~Morf}, {\em Exponentially convergent and numerically efficient solution
  of maxwell’s equations for lamellar gratings}, Journal of the Optical
  Society of America A, 12 (1995), pp.~1043--1056.

\bibitem{Moller1999}
{\sc M.~Möller}, {\em On the unboundedness below of the sturm—liouville
  operator}, Proceedings of the Royal Society of Edinburgh: Section A
  Mathematics, 129 (1999), p.~1011–1015.

\bibitem{Naimark1968Book}
{\sc M.~A. Naimark}, {\em Linear Differential Operators. Vol. II}, Ungar, 1968.
\newblock classic; verify publisher/edition.

\bibitem{Petit2013Book}
{\sc R.~Petit, L.~Botten, M.~Cadilhac, G.~Derrick, D.~Maystre, R.~McPhedran,
  M.~Neviere, and P.~Vincent}, {\em Electromagnetic Theory of Gratings}, Topics
  in Current Physics, Springer Berlin Heidelberg, 2013.

\bibitem{Popov2004a}
{\sc E.~Popov, B.~Chernov, M.~Nevi{\`e}re, and N.~Bonod}, {\em Differential
  theory: application to highly conducting gratings}, Journal of the Optical
  Society of America A, 21 (2004), pp.~199--206.

\bibitem{Pyatkov1992}
{\sc S.~G. Pyatkov}, {\em Certain properties of eigenfunctions of linear
  pencils}, Mathematical Notes, 51 (1992), pp.~90--95.

\bibitem{Qi2014}
{\sc J.~Qi and S.~Chen}, {\em A priori bounds and existence of non-real
  eigenvalues of indefinite sturm--liouville problems}, Journal of Spectral
  Theory, 4 (2014), pp.~53--63.

\bibitem{Qi2016}
{\sc J.~Qi, B.~Xie, and S.~Chen}, {\em The upper and lower bounds on non-real
  eigenvalues of indefinite sturm-liouville problems}, Proceedings of the
  American Mathematical Society, 144 (2016), pp.~547--559.

\bibitem{Randriamihaja2016}
{\sc M.~H. Randriamihaja, G.~Granet, K.~Edee, and K.~Raniriharinosy}, {\em
  Polynomial modal analysis of lamellar diffraction gratings in conical
  mounting}, 2016.

\bibitem{Richardson1918}
{\sc R.~G. Richardson}, {\em Contributions to the study of oscillation
  properties of the solutions of linear differential equations of the second
  order}, American Journal of Mathematics, 40 (1918), pp.~283--316.

\bibitem{Sheng1982}
{\sc P.~Sheng, R.~S. Stepleman, and P.~N. Sanda}, {\em Exact eigenfunctions for
  square-wave gratings: Application to diffraction and surface-plasmon
  calculations}, Phys. Rev. B, 26 (1982), pp.~2907--2916.

\bibitem{Snyder1983Book}
{\sc A.~W. Snyder and J.~D. Love}, {\em Optical Waveguide Theory}, Chapman and
  Hall, London, 1983.

\bibitem{Song2013}
{\sc D.~Song and Y.~Y. Lu}, {\em Pseudospectral modal method for conical
  diffraction of gratings}, Journal of Modern Optics, 60 (2013),
  pp.~1729--1734.

\bibitem{Song2011}
{\sc D.~Song, L.~Yuan, and Y.~Y. Lu}, {\em Fourier-matching pseudospectral
  modal method for diffraction gratings}, Journal of the Optical Society of
  America A, 28 (2011), pp.~613--620.

\bibitem{Sturman2007b}
{\sc B.~Sturman, E.~Podivilov, and M.~Gorkunov}, {\em Eigenmodes for
  metal-dielectric light-transmitting nanostructures}, Physical Review
  B—Condensed Matter and Materials Physics, 76 (2007), p.~125104.

\bibitem{Sturman2007a}
{\sc B.~Sturman, E.~Podivilov, and M.~Gorkunov}, {\em Eigenmodes for the
  problem of extraordinary light transmission through subwavelength holes},
  Europhysics Letters, 80 (2007), p.~24002.

\bibitem{Tishchenko2005}
{\sc A.~V. Tishchenko}, {\em Phenomenological representation of deep and high
  contrast lamellar gratings by means of the modal method}, Optical and Quantum
  Electronics, 27 (2005), pp.~309--330.

\bibitem{Volkmer1996}
{\sc H.~Volkmer}, {\em Sturm--liouville problems with indefinite weights and
  everitt's inequality}, Proceedings of the Royal Society of Edinburgh Section
  A: Mathematics, 126 (1996), pp.~1097--1112.

\bibitem{Watanabe2006}
{\sc K.~Watanabe}, {\em Study of the differential theory of lamellar gratings
  made of highly conducting materials}, J. Opt. Soc. Am. A, 23 (2006),
  pp.~69--72.

\bibitem{Weidmann1987Book}
{\sc J.~Weidmann}, {\em Spectral theory of ordinary differential operators},
  vol.~1258, Springer, 1987.

\bibitem{Wexler1967}
{\sc A.~Wexler}, {\em Solution of waveguide discontinuities by modal analysis},
  IEEE Transactions on Microwave Theory and Techniques, 15 (1967),
  pp.~508--517.

\bibitem{Xie2013}
{\sc B.~Xie and J.~Qi}, {\em Non-real eigenvalues of indefinite
  sturm--liouville problems}, Journal of Differential Equations, 255 (2013),
  pp.~2291--2301.

\bibitem{Xie2017}
{\sc B.~Xie, H.~Sun, and X.~Guo}, {\em Non-real eigenvalues of symmetric
  sturm--liouville problems with indefinite weight functions}, Electron. J.
  Qual. Theory Differ. Equ, 18 (2017), pp.~1--14.

\bibitem{Zettl2005Book}
{\sc A.~Zettl}, {\em Sturm-Liouville Theory}, Mathematical surveys and
  monographs, American Mathematical Society, 2005.

\end{thebibliography}

\end{document}